\documentclass[11pt, reqno]{amsart}
\usepackage[all]{xy}
\usepackage{a4wide,fullpage}  \usepackage{amssymb}
\usepackage{amsthm}
\usepackage{amsmath,cancel}
\usepackage{amscd,enumitem}
\usepackage{verbatim}
\usepackage{float}
\usepackage{color,mdframed}
\usepackage{bbm}
\usepackage[mathscr]{eucal}
\usepackage[all]{xy}
\usepackage{bm}
\usepackage[draft]{hyperref}
\usepackage{calligra}
\DeclareMathAlphabet{\mathcalligra}{T1}{calligra}{m}{n}

\theoremstyle{plain}
\newtheorem{theorem}{Theorem}
\newtheorem{corollary}[theorem]{Corollary}
\newtheorem{lemma}[theorem]{Lemma}

\newtheorem*{theorem*}{Theorem}

\theoremstyle{definition}
\newtheorem{definition}[theorem]{Definition}

\theoremstyle{remark}
\newtheorem*{remark}{Remark}

\newcommand{\R}{\mathbb{R}}

\newcommand{\Z}{\mathbb{Z}}
\newcommand{\N}{\mathbb{N}}
\newcommand{\C}{\mathbb{C}}

\renewcommand{\H}{\mathbb{H}}

\newcommand{\ord}{{\text {\rm ord}}}
\newcommand{\z}{\mathfrak{z}}
\newcommand{\y}{\mathbbm{y}}
\newcommand{\x}{\mathbbm{x}}
\newcommand{\w}{\mathbbm{w}}
\newcommand{\f}{\mathbbm{f}}

\newcommand{\calJ}{\mathcal{J}}

\newcommand{\GG}{\mathbb{G}}
\newcommand{\JJ}{\mathbb{J}}
\newcommand{\jj}{\mathbbm{j}}
\newcommand{\bgg}{\mathbbm{g}}

\newcommand{\W}{\mathcal{W}}
\newcommand{\HH}{\mathbb{H}}
\newcommand{\II}{\mathbb{I}}

\newcommand{\Log}{\operatorname{Log}}
\newcommand{\sgn}{\operatorname{sgn}}

\newcommand{\re}{\operatorname{Re}}
\newcommand{\im}{\operatorname{Im}}

\newcommand{\SL}{{\text {\rm SL}}}

\newcommand{\Ei}{\operatorname{Ei}}
\newcommand{\Ein}{\operatorname{Ein}}

\renewcommand{\Re}{\operatorname{Re}}
\renewcommand{\pmod}[1]{\  \,  \left(  \operatorname{mod} \,  #1 \right)}

\DeclareMathAlphabet{\mathpzc}{OT1}{pzc}{m}{it}

\numberwithin{equation}{section}
\numberwithin{theorem}{section}
\allowdisplaybreaks

\begin{document}
\title{An extension of Rohrlich's Theorem to the $j$-function}
\author{Kathrin Bringmann}
\address{Mathematical Institute, University of Cologne, Weyertal 86-90, D--50931 Cologne, Germany}
\email{kbringma@math.uni-koeln.de}
\author{Ben Kane}
\address{Department of Mathematics, University of Hong Kong, Pokfulam, Hong Kong}
\email{bkane@hku.hk}
\thanks{The research of the  first author is supported by the Alfried Krupp Prize for Young University Teachers of the Krupp foundation and the research leading to these results receives funding from the European Research Council under the European Union's Seventh Framework Programme (FP/2007-2013) / ERC Grant agreement n. 335220 - AQSER. The research of the second author was supported by grants from the Research Grants Council of the Hong Kong SAR, China (project numbers HKU 17302515, 17316416, 17301317, and 17303618).}
\maketitle
\section{Introduction and statement of results}\label{sec:introduction}
We start by recalling the following theorem of Rohrlich \cite{Rohrlich}. To state it, let $\omega_{\z}$ denote half of the size of the stabilizer $\Gamma_{\z}$ of $\z\in\H$ in $\SL_2(\Z)$ and for a meromorphic function $f:\mathbb H\to\C$ let $\ord_{\z}(f)$ be the order of vanishing of $f$ at $\z$. Moreover define $\Delta(z):=q\prod_{n\geq 1} (1-q^n)^{24}$, where $q:=e^{2\pi iz}$, and set $\jj(z):=\frac{1}{6}\log(y^6|\Delta(z)|)+1$, where $z=x+iy$.
Rohrlich's Theorem may be stated in terms of
the Petersson inner product, denoted by $\langle\ ,\, \rangle$.

\begin{theorem}[Rohrlich \cite{Rohrlich}]\label{thm:Rohrlich}
Suppose that $f$ is a meromorphic modular function with respect to $\SL_2(\Z)$ that
does not have a pole at $i\infty$ and
has constant term one in its Fourier expansion. Then
\[
\left\langle 1,\log|f|\right\rangle = -2\pi\sum_{\z\in\SL_2(\Z)\backslash\H} \frac{\ord_\z(f)}{\omega_\z}\jj(\z).
\]
\end{theorem}
\begin{remark}
In \cite{Rohrlich}, Theorem \ref{thm:Rohrlich} was stated for $\jj-1$ instead. However, by the valence formula, these two statements are equivalent.

\end{remark}
The function $\jj$ is a weight zero sesquiharmonic Maass form i.e., it is invariant under the action of $\SL_2(\Z)$ and it is annihilated by $\xi_0\circ \xi_2 \circ \xi_0$, where $\xi_{\kappa}:=2i y^{\kappa} \overline{\frac{\partial}{\partial \overline{z}}}$ (see Section \ref{subspolar} for a full definition). More precisely, $\Delta_{0}(\jj)=1$, where $\Delta_{\kappa}:=-y^2(\frac{\partial^2}{\partial x^2}+\frac{\partial^2}{\partial y^2}) +i\kappa y (\frac{\partial}{\partial x} + i\frac{\partial}{\partial y})$ satisfies $\Delta_{\kappa}=-\xi_{2-\kappa}\circ\xi_{\kappa}$.

To extend this, let $j_1:=j-744$, with $j$ the usual $j$-invariant and set $j_n:=j_1|T_n$, where for a function $f$ transforming of weight $\kappa$, we define the $n$-th \begin{it}Hecke operator\end{it} by
\[
f|T_n (z):=\sum_{\substack{ad=n\\ d>0}} \ \ \sum_{b\pmod{d}} d^{-\kappa} f\left(\tfrac{az+b}{d}\right).
\]
There are functions $\jj_n$ defined in \eqref{eqn:Jxi} below whose properties are analogous to those of $\jj_0:=\jj$
if we define $j_0:=1$. Namely, these functions are weight zero sesquiharmonic Maass forms
that satisfy $\Delta_{0}\left(\jj_n\right) = j_n$
and are furthermore chosen uniquely so that the principal parts of their Fourier and elliptic expansions essentially only contain a single term which maps to the principal part of $j_n$ under $\Delta_0$. More precisely, they have a purely sesquiharmonic principal part, up to a possible constant multiple of $y$, vanishing constant terms in their Fourier expansion, and a trivial principal part in their elliptic expansions around every point in $\H$; see Lemmas \ref{lem:Fourierexps} and \ref{lem:ellexps} below for the shape of their Fourier and elliptic expansions, respectively. In addition, they also satisfy the following extension of Theorem \ref{thm:Rohrlich}. Here we use a regularized version of the inner product (see \eqref{eqn:innerdef} below), which we again denote by $\langle\ ,\, \rangle$. This regularization was first introduced by Petersson in \cite{Pe2} and then later independently rediscovered and generalized by Borcherds \cite{Bo1} and Harvey--Moore \cite{HM}.

\begin{theorem}\label{thm:jninner}
Suppose that $f$ is a meromorphic modular function with respect to $\SL_2(\Z)$ which has constant term one in its Fourier expansion. Then
\begin{equation*}
\left\langle j_n,\log|f|\right\rangle=-2\pi\sum_{\z\in \SL_2(\Z)\backslash\H}\frac{\ord_{\z}(f)}{\omega_{\z}}\jj_n(\z).
\end{equation*}
\end{theorem}
\begin{remark}
This research was motivated by generalizations of Rohrlich's Theorem in other directions, such as the recent work of Herrero, Imamo$\overline{\text{g}}$lu, von Pippich, and T\'oth \cite{HIvPT}.
\end{remark}
Theorem \ref{thm:Rohrlich} was also generalized by Rohrlich \cite{Rohrlich} by replacing the meromorphic function $f$ in Theorem \ref{thm:Rohrlich} with a meromorphic modular form of weight $k$ times $y^{\frac{k}{2}}$, yielding again a weight zero object. We similarly
extend Theorem \ref{thm:jninner} in such a direction.

\begin{theorem}\label{thm:jninnergen}
There exists a constant $c_n$ such that for every weight $k$ meromorphic modular form $f$ with respect to $\SL_2(\Z)$ that does not have a pole at $i\infty$ and has constant term one in its Fourier expansion, we have
\[
\left<j_n,\log\!\left(y^{\frac{k}{2}}|f|\right)\right>=-2\pi\sum_{\z\in \SL_2(\Z)\backslash\H}\frac{\ord_{\z}(f)}{\omega_{\z}}\jj_n(\z)+\frac{k}{12}c_n.
\]
\end{theorem}
\begin{remark}
Plugging $k=0$ into Theorem \ref{thm:jninnergen}, we see that Theorem \ref{thm:jninner} is an immediate corollary.
\end{remark}

An interesting special case of Theorem \ref{thm:jninner} arises if one takes $f$ to be a so-called \begin{it}prime form\end{it}, which is a modular form which vanishes at precisely one point $\z\in\H$ and has a simple zero at $\z$ (see \cite[Section 1.c]{Pe7} for a full treatment of these functions). By the valence formula, the prime forms necessarily have weight $k=12 \omega_{\z}^{-1}$ and may  directly be computed as
\[
\left(\Delta(z)\Big(j(z)-j(\z)\Big)\right)^{\frac{1}{\omega_{\z}}}.
\]
Multiplying by $y^{\frac{k}{2}}$ and taking the logarithm of the absolute value, it is hence natural to consider the functions
\begin{equation}\label{eqn:primeformgdef}
\bgg_{\z}(z):=\log\left(y^6\left|\Delta(z)\Big(j(z)-j(\z)\Big)\right|\right),
\end{equation}
and Theorem \ref{thm:jninner} states that
\begin{equation}\label{eqn:jnbggz}
\left<j_n,\bgg_{\z}\right> = -2\pi\jj_n(\z) + c_{n}.
\end{equation}
When characterizing modular forms via their divisors, the prime forms are natural building blocks because they vanish at precisely one point in $\H$, allowing one to easily construct a function with a given order of vanishing at each point. In the same way, since each function $\bgg_{\z}$ appearing on the left-hand side of \eqref{eqn:jnbggz} has a singularity at only one point and the single term $\jj_n(\z)$ is isolated on the right-hand side of \eqref{eqn:jnbggz}, it is natural to use the functions $\bgg_{\z}$ as building blocks for the logarithms of weight $k$ meromorphic modular forms.

\begin{remark}
If one were only interested in proving Theorem \ref{thm:jninner}, then one could choose the building blocks $z\mapsto\log|j(z)-j(\z)|$ instead of $\bgg_{\z}$.
However, as noted above, the functions $\bgg_{\z}$ are more natural when considering divisors of modular forms
because they only have a singularity precisely at the point $\z$, while the functions $z\mapsto \log|j(z)-j(\z)|$ have a singularity both at $\z$ and $i\infty$.
\end{remark}

Generating functions of traces of $\SL_2(\Z)$-invariant objects such as $j_n$ have a long history going back to the paper of Zagier on traces of singular moduli \cite{ZagierSingular}. To give a related example, let $\mathcal{Q}_D$ denote the set of integral binary quadratic forms of discriminant $D$. The generating function, with $\tau_Q\in\H$ the unique root of $Q(z,1)$,
\begin{equation}\label{eqn:tracejj}
\sum_{\substack{D<0\\ D\equiv 0,1\pmod{4}}} \sum_{Q\in \mathcal{Q}_D/\SL_2(\Z)} \jj\!\left(\tau_Q\right) e^{2\pi i |D|\tau}
\end{equation}
was shown by Bruinier and Funke \cite[Theorem 1.2]{BruinierFunkeTraces} to be the holomorphic part of a weight $\frac32$ modular object. Instead of taking the generating function in $D$, one may also sum in $n$ to obtain, for $y$ sufficiently large,
\begin{equation}\label{Hz}
H_{\z}(z):=\sum_{n\geq 0} j_n(\z) q^n.
\end{equation}
This function was shown by Asai, Kaneko, and Ninomiya \cite{AKN} to satisfy the identity
\[
H_{\z}(z)=-\frac{1}{2\pi i} \frac{j_1'(z)}{j_1(z)-j_1(\z)}.
\]
This identity is equivalent to the denominator formula
\[
j_1(\z)-j_1(z)=e^{-2\pi i\z} \prod_{m\in\N,\,n\in \Z}\left(1-e^{2\pi i m \z}e^{2\pi i n z}\right)^{c(mn)},
\]
for the Monster Lie algebra, where $c(m)$ denotes the $m$-th Fourier coefficient of $j_1$. The function $H_{\z}$ is a weight two meromorphic modular form with a simple pole at $z=\z$. For a meromorphic modular form $f$ which does not vanish at $i\infty$, it is then natural to define the \begin{it}divisor modular form\end{it}
\[
f^{\operatorname{div}}(z):=\sum_{\z\in\SL_2(\Z)\backslash\H}\frac{\ord_{\z}(f)}{\omega_{\z}} H_{\z}(z).
\]
Bruinier, Kohnen, and Ono \cite[Theorem 1]{BKO} showed that if $f$ satisfies weight $\kappa$ modularity then $f^{\operatorname{div}}$ is related to the logarithmic derivative of $f$ via
\[
f^{\operatorname{div}}=-\frac{1}{2\pi i} \frac{f'}{f} +\frac{\kappa}{12}E_2,
\]
where $E_2$ denotes the quasimodular weight two Eisenstein series.
Analogously to \eqref{Hz}, we define the generating function, for $y$ sufficiently large,
\[
\HH_{\z}(z):=\sum_{n\geq 0}\jj_n(\z) q^n,
\]
and its related divisor modular form
\[
\f^{\operatorname{div}}(z):=\sum_{\z\in\SL_2(\Z)\backslash\H}\frac{\ord_{\z}(f)}{\omega_{\z}}\HH_{\z}(z).
\]

The function $\HH_{\z}$ turns out to be the holomorphic part of a weight two sesquiharmonic Maass form, while the generating function
\begin{align*}
 \II_\z (z):=\sum_{n\geq 0}\langle \bgg_{\z},j_n\rangle q^n,
\end{align*}
which is closely related by Theorem \ref{thm:jninnergen}, is the holomorphic part of a biharmonic Maass form. A weight $\kappa$ biharmonic Maass form satisfies weight $\kappa$ modularity and is annihilated by $\Delta_{\kappa}^2=\left(\xi_{2-\kappa}\circ \xi_{\kappa}\right)^2$ (see Section \ref{sec:construction} for a full definition).
\begin{theorem}\label{thm:Fdivgen}
The function $\HH_\z$ is the holomorphic part of a weight two sesquiharmonic Maass form $\widehat{\HH}_{\z}$ and $\II_{\z}$ is the holomorphic part of a weight two biharmonic Maass form $\widehat{\II}_{\z}$.
\end{theorem}
\begin{remark}
Consider
\[
\Theta(z,\tau):=\sum_{n\geq 0} \sum_{\substack{D<0\\ D\equiv 0,1\pmod{4}}} \sum_{Q\in \mathcal{Q}_D/\SL_2(\Z)} \jj_n\!\left(\tau_Q\right) e^{2\pi i |D|\tau} e^{2 \pi i nz}.
\]
The modularity of \eqref{eqn:tracejj} hints that $\tau\mapsto\Theta(z,\tau)$ may have a relation to a function satisfying weight $\frac32$ modularity, while  we see in Theorem \ref{thm:Fdivgen} that it is the holomorphic part of a weight two object as a function of $z$ (assuming convergence). It should be possible to prove this modularity using the methods from \cite{BES} (and which in particular requires generalizing Proposition 1.3 of \cite{BES} to include functions which have poles in points of the upper half plane).
\end{remark}
As a corollary to Theorem \ref{thm:Fdivgen}, one obtains that $\f^{\operatorname{div}}$ has the modular completion
\[
\widehat{\f}^{\operatorname{div}}(z):=\sum_{\z\in\SL_2(\Z)\backslash\H}\frac{\ord_{\z}(f)}{\omega_{\z}}\widehat{\HH}_{\z}(z).
\]
\begin{corollary}\label{cor:Fdiv}
If $f$ is a meromorphic modular function with constant term one in its Fourier expansion, then the function $\f^{\operatorname{div}}$ is the holomorphic part of the Fourier expansion of a weight two sesquiharmonic Maass form $\widehat{\f}^{\operatorname{div}}$. Moreover we have
\[
\xi_{2}\left(\widehat{\f}^{\operatorname{div}}\right)=-\frac{1}{2\pi}\log|f|.
\]
\end{corollary}

The paper is organized as follows. In Section \ref{sec:specialfunctions}, we introduce some special functions. In Section \ref{sec:construction}, we construct a number of functions and discuss their properties. In Section \ref{sec:FourierElliptic}, we determine the shapes of the Fourier and elliptic expansions of the functions defined in Section \ref{sec:construction}. In Section \ref{sec:jninner}, we compute inner products in order to prove Theorem \ref{thm:jninnergen}.
 In Section \ref{sec:Fdiv}, we consider the generating functions of these inner products and prove Theorem \ref{thm:Fdivgen} and Corollary \ref{cor:Fdiv}.
\section*{Acknowledgements}
The authors thank Claudia Alfes-Neumann, Stephan Ehlen, and Markus Schwagenscheidt for helpful comments on an earlier version of this paper and the referee for carefully reading our paper.

\section{Special functions, Poincar\'e series, and polyharmonic Maass forms}\label{sec:specialfunctions}
\subsection{The incomplete gamma function and related functions}
We use the principal branch of the complex logarithm, denoted by $\Log$, with the convention that, for $w>0$,
$$
\Log(-w)=\log(w)+\pi i,
$$
where $\log: \R^+ \to \R$ is the natural logarithm.

For $s,w \in\mathbb{C}$ with $\Re(w)>0$, define the {\it generalized exponential integral} $E_s$ (see \cite[8.19.3]{NIST}) by
\begin{equation*}
E_s(w) := \int_1^{\infty} e^{-wt}t^{-s}\, dt.
\end{equation*}
This function is related to the incomplete Gamma function,
defined for $\re(s)>0$ and $w\in\C$ by
\[
\Gamma(s,w) := \int_{w}^\infty e^{-t}t^{s}\, \frac{dt}{t},
\]
via (see \cite[8.19.1]{NIST})
\begin{equation}\label{eqn:GammaEr}
\Gamma(s, w) = w^s E_{1-s}(w).
\end{equation}
Up to a branch cut along the negative real line, the function $E_1$ may be analytically continued via
\begin{equation*}
E_1(w) = \Ein(w) -  \Log(w) - \gamma
\end{equation*}
(see \cite[6.2.4]{NIST}), where $\gamma$ is the Euler-Mascheroni constant and $\Ein$ is the entire function given by
\begin{equation*}
\Ein(w) :=\int_0^{w} \left(1-e^{-t}\right)\frac{dt}{t}=\sum_{n\geq 1}\frac{(-1)^{n+1}}{n!\, n} w^n.
\end{equation*}
The function $\Ein$ also appears in the definition of the exponential integral, namely
\begin{equation*}
\Ei(w):=-\Ein(-w)+\Log(w)+\gamma.
\end{equation*}
In order to obtain a function which is real-valued for $w\in\R\setminus\{0\}$, for $\kappa\in\Z$ it is natural to define
\begin{equation*}
W_\kappa(w) := (-2w)^{1-\kappa} \Re(E_\kappa(-2w)).
\end{equation*}
By \cite[Lemma 2.4]{BDE}, \eqref{eqn:GammaEr}, and the fact that $E_\kappa(w)\in\R$ for $w>0$, we obtain the following.
\begin{lemma}\label{lem:relationsofnonhol}
	For $w>0$, we have
	\begin{equation*}
	W_\kappa(w)  = (-1)^{1-\kappa}\left((2w)^{1-\kappa}E_\kappa(-2w) + \frac{\pi i}{\Gamma(\kappa)}\right)
	= \Gamma(1-\kappa, -2w)+\frac{(-1)^{1-\kappa}\pi i}{(\kappa-1)!}.
	\end{equation*}
	For $w<0$, we have
	\begin{equation*}
	W_\kappa(w) = \Gamma(1-\kappa,-2w).
	\end{equation*}
\end{lemma}

We next define
\begin{equation*}
\bm{\W}_s(w):=\int\limits_{\sgn(w)\infty}^w W_{2-s}(-t) t^{-s}e^{2t}dt.
\end{equation*}

A direct calculation shows the following lemma.
\begin{lemma}\label{lem:xiFourier}
For $m\in \Z \backslash \{0\}$ we have
\begin{align*}
\xi_{\kappa}\!\left(W_\kappa(2\pi my)q^m\right)&=
-(-4\pi m)^{1-\kappa}
q^{-m},\qquad \xi_\kappa\!\left(\bm{\W}_{\kappa}(2\pi my)q^m\right)=(2\pi m)^{1-\kappa} W_{2-\kappa}(-2\pi my) q^{-m}.
\end{align*}
In particular, $W_\kappa(2\pi my)q^m$ is annihilated by $\Delta_{\kappa}$, and $\bm{\W}_{\kappa}(2\pi my)q^m$ is annihilated by $\xi_{\kappa}\circ\Delta_{\kappa}$.
\end{lemma}
We next determine the asymptotic behavior for $W_\kappa$ and $\bm{\W}_\kappa$.
\begin{lemma}\label{lem:Wgrowth}
As $w\to \pm \infty$, we have
\begin{align*}
W_\kappa(w)&=
(-2w)^{-\kappa}
 e^{2w}+O\!\left(w^{-\kappa-1}e^{2w}\right),\qquad \bm{\W}_\kappa(w)=
-2^{\kappa-2}
w^{-1}+O\!\left(w^{-2}\right).
\end{align*}
\end{lemma}
\begin{proof}
It is not hard to conclude the claims from (see \cite[8.11.2]{NIST})
\[
\Gamma(1-\kappa,-2w)=
(-2w)^{-\kappa}
e^{2w}\!\left(1+O\!\left(w^{-1}\right)\right). \qedhere
\]
\end{proof}

\subsection{Polar polyharmonic Maass forms}\label{subspolar}
In this section, we introduce polar polyharmonic Maass forms. Letting $\xi_{\kappa}^{\ell}$ denote the $\xi$-operator repeated $\ell$ times, a \begin{it}polar polyharmonic Maass form on $\SL_2(\Z)$ of weight $\kappa\in 2\Z$ and depth $\ell\in\N_0$\end{it} is a function $F:\H\to\C$ satisfying the following:
\noindent

\noindent
\begin{enumerate}[leftmargin=*, label={\rm(\arabic*)}]
\item For every $\gamma=\left(\begin{smallmatrix} a&b\\c&d\end{smallmatrix}\right)\in\SL_2(\Z)$, we have
\[
F|_{
\kappa
}\gamma(z):=(cz+d)^{-
\kappa
} F\left(\tfrac{az+b}{cz+d}\right) = F(z).
\]
\item The function $F$ is annihilated by $\xi_{
\kappa
}^{\ell}$.

\item For each $\z\in\H$, there exists an $m_{\z}\in\N_0$ such that $\lim_{z\to\z}(r_{\z}^{m_{\z}}(z)F(z))$ exists, where
\begin{equation}\label{eqn:rXdef}
r_{\z}(z):=|X_{\z}(z)|\quad \text{ with }X_{\z}(z):=\tfrac{z-\z}{z-\overline{\z}}.
\end{equation}
\item The function $F$ grows at most linear exponentially as $z\to i\infty$.

\end{enumerate}
\begin{remark}
	Note that in (2), $\xi_\kappa^\ell$ can be written in terms of $\Delta_\kappa$ if $\ell$ is even. If a function satisfies (2) (but is not necessarily modular), then we simply call it \begin{it}depth $\ell$ with respect to the weight $\kappa$.\end{it}
\end{remark}
\begin{remark}
	Note that our notation differs from that in \cite{LR}.
In particular, if the depth is $\ell$ in this paper, then it is depth $\frac{\ell}{2}$ in \cite{LR}.
\end{remark}
We omit the adjective ``polar'' whenever the only possible singularity occurs at $i\infty$. Note that polar polyharmonic Maass forms of depth one are meromorphic modular forms. We call a polar polyharmonic Maass form $F$ of depth two a \begin{it}polar harmonic Maass form\end{it} and those of depth three are \begin{it}polar sesquiharmonic Maass forms\end{it}. We call those forms of depth four \begin{it}biharmonic\end{it}.

\subsection{Niebur Poincar\'e series}\label{sec:seeds}

We next recall the Niebur Poincar\'e series \cite{Niebur}
\[
F_{m}(z,s):=\sum_{\gamma\in\Gamma_{\infty}\backslash\SL_2(\Z)} \varphi_{m,s}(\gamma z), \quad  (\Re(s)>1, m \in \Z)
\]
where $\Gamma_\infty:=\{\pm \left(\begin{smallmatrix} 1 & n \\ 0 & 1\end{smallmatrix}\right)\,:\, n\in\Z\}$ and
\begin{equation*}
\varphi_{m,s}(z):= y^{\frac{1}{2}} I_{s-\frac{1}{2}}(2\pi|m|y)e^{2\pi i mx}.
\end{equation*}
Here $I_{\kappa}$ is the $I$-Bessel function of order $\kappa$. The functions $s\mapsto F_{m}(z,s)$ have meromorphic continuations to $\C$ and do not have poles at $s=1$ \cite[Theorem 5]{Niebur}. The functions $\varphi_{m,s}$ are eigenfunctions under $\Delta_0$ with eigenvalue $s(1-s)$. Hence for any $s$ with $\re(s)$ sufficiently large so that $z\mapsto F_{m}(z,s)$ converges absolutely and locally uniformly, we conclude that $F_{m}(z,s)$ is also an eigenfunction under $\Delta_0$ with eigenvalue $s(1-s)$. Arguing by meromorphic continuation, we obtain that
\[
\Delta_{0}\!\left(F_{m}(z,s)\right)=s(1-s)F_m(z,s)
\]
for any $s$ for which $F_{m}(z,s)$ does not have a pole. In particular, one may use this to construct harmonic Maass forms by taking $s=1$. Indeed, by \cite[Theorem 6]{Niebur} (note that there is a missing $2\pi \sqrt{n}$), there exists a constant $\mathcal C_n$ such that
\begin{equation}\label{eqn:Fnjn}
2\pi\sqrt{n}F_{-n}(z,1)=j_n(z)+\mathcal C_n.
\end{equation}

\subsection{Real-analytic Eisenstein series}

Throughout the paper, we use various properties of the real-analytic Eisenstein series, defined for $\re(s)>1$ by
\[
E(z,s):=\sum_{\gamma\in\Gamma_{\infty}\backslash\SL_2(\Z)} \im(\gamma z)^s.
\]
Via the Hecke trick, $E(z,s)$ is closely related to the weight two completed Eisenstein series
\[
\widehat{E}_2(z):=1-24\sum_{m\geq 1} \sigma_1(m) q^m -\frac{3}{\pi y},
\]
where $\sigma_\ell(n):=\sum_{d|n}d^\ell$. The following properties of $E(z,s)$ and $\widehat{E}_2(z)$ are well-known.
\begin{lemma}\label{lem:Eprop}
\noindent

\noindent
\begin{enumerate}[leftmargin=*, label={\rm(\arabic*)}]
\item
The function $s\mapsto E(z,s)$ has a meromorphic continuation to $\C$ with a simple pole at $s=1$ of residue $\frac{3}{\pi}$.
\item
The function $z\mapsto E(z,s)$ is an eigenfunction with eigenvalue $s(1-s)$ under $\Delta_0$.
\item
The function $\widehat{E}_2$ is a weight two harmonic Maass form which satisfies
\[
\xi_{2}\!\left(\widehat{E}_2\right) = \frac{3}{\pi}.
\]
\item
Denoting by $\operatorname{CT}_{s=s_0}(f(s))$ the constant term in the Laurent expansion of the analytic continuation of a function $f$ around $s=s_0$,
we have
\[
\widehat{E}_2(z)=\operatorname{CT}_{s=1} \left(\xi_{0}\!\left(E(z,s)\right)\right).
\]
\end{enumerate}
\end{lemma}

In light of Lemma \ref{lem:Eprop} (1), it is natural to define, for some $\mathcal{C}\in\C$,
 \begin{equation}\label{defEz}
\mathcal{E}(z):=\lim_{s\to 1} \left(4\pi E(z,s)-\tfrac{12}{s-1}\right)+\mathcal{C}.
\end{equation}
Letting $\zeta(s)$ denote the Riemann zeta function, we specifically choose $\mathcal{C}:=-24\gamma +24\log(2)+144\frac{\zeta'(2)}{\pi^{2}}$ so that, by  \cite[Section II, (2.17) and (2.18)]{GrossZagier},
\begin{equation}\label{eqn:calEgrowth}
\lim\limits_{y \to \infty}\left(\mathcal{E}(z)-4\pi y +12\log(y)\right)=0.
\end{equation}

\section{Construction of the functions}\label{sec:construction}

In this section, we construct two weight two biharmonic functions $\GG_{\z}$ and $\JJ_n$ satisfying
\begin{equation}\label{eqn:GJxiops}
{
\xi_2(\GG_{\z})=\bgg_{\z},
}
%\implies \xi_2\circ\xi_0\circ\xi_2(\GG_\z)=-6,
 \qquad \xi_2\circ\xi_0\circ\xi_2(\JJ_n)=-j_n,
\end{equation}
where $\bgg_{\z}$ is the weight zero sesquiharmonic Maass form
that is defined in \eqref{eqn:primeformgdef}.

 Similarly, $\JJ_n$ is constructed to have a singularity at $i\infty$ such that $\xi_{2}(\JJ_n)$ is the function $\jj_n$ given in the introduction, i.e., the only singularity of $\JJ_n$ lies in its biharmonic part.

Recall the \begin{it}automorphic Green's function\end{it}
\begin{equation}\label{eqn:Greensdef}
G_s(z,\z):=\sum_{\gamma\in\SL_2(\Z)} g_s\!\left(z,\gamma \z\right).
\end{equation}
Here, with $Q_s$ the Legendre function of the second kind, we define
\[
g_s(z,\z):=-2Q_{s-1}\!\left(1+\tfrac{|z-\z|^2}{2y\y}\right), \quad  \z=\x+i\y\in\H.
\]
The series \eqref{eqn:Greensdef} is convergent for $\re(s)>1$, is an eigenfunction under $\Delta_0$ with eigenvalue $s(1-s)$, and has a meromorphic continuation to the whole $s$-plane (see \cite[Chapter 7, Theorem 3.5]{Hejhal}). For $\gamma\in\SL_2(\R)$, $g_s$ satisfies $g_{s}(\gamma z,\gamma\z)=g_{s}(z,\z)$, yielding that $G_{s}(z,\z)$ is  $\SL_2(\Z)$-invariant in $z$ and $\z$.
\begin{remark}
Using \cite[Proposition 5.1]{GZ2} and the Kronecker limit formula, $G_s(z,\z)$ is related to $\bgg_{\z}$ by
\begin{equation}\label{eqn:Gzdef}
\bgg_{\z}(z)=\frac{1}{2}\lim_{s\to 1} \big(G_{s}(z,\z)+4\pi E(\z,s)\big)-12.
\end{equation}
\end{remark}
In the next lemma, we collect some other useful properties of $G_s(z,\z)$.
\begin{lemma}\label{lem:Greensprop}
\noindent

\noindent
\begin{enumerate}[leftmargin=*, label={\rm(\arabic*)}]
\item The function $s \mapsto G_{s}(z,\z)$ has a simple pole at $s=1$ with residue $-12$.
\item The limit in \eqref{eqn:Gzdef} exists.
\item The function $\bgg_{\z}$ is sesquiharmonic with $\Delta_{0}(\bgg_{\z})=6$.
\item The only singularity of $\bgg_{\z}$ in $\SL_2(\Z)\backslash\mathbb H$ is at $z=\z$ with principal part $4\omega_{\z}\log(r_{\z}(z))$.
\end{enumerate}
\end{lemma}
To construct $\GG_{\z}$ and $\JJ_n$, we find natural preimages
of certain polyharmonic Maass forms, such as $\bgg_{\z}$, under the $\xi$-operator. For this we study the Laurent expansions of eigenfunctions under $\Delta_\kappa$.
\begin{lemma}\label{lem:Laurent}

 Suppose that $z\mapsto f(z,s)$ is an eigenfunction with eigenvalue $(s-\frac{\kappa}{2})(1-s-\frac{\kappa}{2})$ under $\Delta_{\kappa}$ and that $s\mapsto f(z,s)$ is meromorphic. Then for $s$ close to $1-\frac{\kappa}{2}$, we have the Laurent expansion
\begin{equation*}
f(z,s)=\sum_{m\gg -\infty} f_{m}(z)\left(s+\tfrac{\kappa}{2}-1\right)^{m}.
\end{equation*}
The coefficients $f_{m}$ have the following properties:
\begin{enumerate}[leftmargin=*, label={\rm(\arabic*)}]
\item
We have
\begin{equation*}
\Delta_{\kappa}\!\left(f_{m}\right) = (\kappa-1)f_{m-1}-f_{m-2}.
\end{equation*}
\item
If $s\mapsto f(z,s)$ is holomorphic at $s=1-\frac{\kappa}{2}$, then $f_0$ (resp. $f_1$) is annihilated by $\Delta_{\kappa}$ (resp. $\Delta_{\kappa}^2$).
\item
If $ z \mapsto f(z,1-\frac{\kappa}{2})$ vanishes identically, then $f_1$ is annihilated by $\Delta_{\kappa}$.
\item
We have
\begin{equation*}
\xi_{2-\kappa}\left(\operatorname{CT}_{s=1-\frac{\kappa}{2}} \!\left(\frac{\partial}{\partial s} \xi_{\kappa}(f(z,\overline{s}))\right)\right) = (1-\kappa) \operatorname{CT}_{s=1-\frac{\kappa}{2}}\!\left(f(z,s)\right) +\operatorname{Res}_{s=1-\frac{\kappa}{2}}\!\left(f(z,s)\right).
\end{equation*}
\end{enumerate}
\end{lemma}
\begin{proof}
(1) The claim follows using the eigenfunction property by comparing coefficients in the Laurent expansions on both sides of Lemma \ref{lem:Laurent} (1).\\
(2) By (1) we have
\begin{equation}\label{eqn:Delf0f1}
\Delta_{\kappa}\!\left(f_{0}\right) = (\kappa-1) f_{-1}-f_{-2},\qquad \Delta_{\kappa}\!\left(f_{1}\right) = (\kappa-1)f_0-f_{-1}.
\end{equation}
Since $s\mapsto f(z,s)$ is holomorphic at $s=1-\frac{\kappa}{2}$, we have that $f_{\ell}=0$ for $\ell<0$, yielding the claim.\\
(3) The claim follows immediately from \eqref{eqn:Delf0f1} and the fact that $f(z,1-\frac{\kappa}{2})=f_0(z)$ if $f(z,s)$ is holomorphic at $s=1-\frac{\kappa}{2}$.\\
(4) Noting that
\begin{equation*}
f_1(z)=\operatorname{CT}_{s=1-\frac{\kappa}{2}} \left(\frac{\partial}{\partial s} f(z,s)\right),\quad
f_0(z)=\operatorname{CT}_{s=1-\frac{\kappa}{2}} \!\left(f(z,s)\right),\quad
f_{-1}(z)=\operatorname{Res}_{s=1-\frac{\kappa}{2}} \!\left(f(z,s)\right),
\end{equation*}
it is not hard to conclude the statement using \eqref{eqn:Delf0f1} and the fact that $1-\frac{\kappa}{2}$ is real so that
\[
\operatorname{CT}_{\overline{s}=1-\frac{\kappa}{2}}\left(\frac{\partial}{\partial\overline{s}}\xi_{\kappa}(f(z,s))\right)=\operatorname{CT}_{s=1-\frac{\kappa}{2}}\left(\frac{\partial}{\partial s}\xi_{\kappa}(f(z,\overline{s}))\right). \qedhere
\]
\end{proof}

 We are now ready to construct the functions $\GG_{\z}$ and $\JJ_n$.
Namely, we set
\begin{align}
\label{eqn:GGdef}
\GG_{\z}(z)&:=\frac{1}{2}
\GG_{\z,1}(z)+ \frac{\pi}{6} \left(\overline{\mathcal{E}(\z)}-\overline{\mathcal{C}}-12\right)\widehat{E}_2(z), \ \text{where} \ \GG_{\z,s}(z):=\frac{\partial}{\partial s}\xi_0\!\left(G_{\overline{s}}(z,\z)\right),\\
\label{eqn:scrJndef}
\JJ_n(z)&:= -\pi\sqrt{n}\mathcal{F}_{-n}(z,1) +\frac{\overline{\mathcal{C}_n}}{12}\mathbb{E}(z) + a_n \widehat{E}_2(z),
\end{align}
where
\begin{align}
\mathbb{E}(z)&:=4\pi \left[\frac{\partial}{\partial s} \xi_0(E(z,\overline{s}))\right]_{s=1}+\frac{\pi}{3}\left(\overline{\mathcal{C}}-12\right) \widehat{E}_{2}(z)\notag,\\
\label{eqn:mathcalFdef}
\mathcal{F}_{-n}(z,s)&:=\frac{\partial^2}{\partial s^2} \xi_0\!\left(F_{-n}(z,\overline{s})\right)-2\frac{\partial}{\partial s} \xi_0\!\left(F_{-n}(z,\overline{s})\right)
\end{align}
with $\mathcal{C}$ and $\mathcal{C}_n$ given in \eqref{defEz} and \eqref{eqn:Fnjn}, respectively, and with $a_n$ determined below in
Subsection \ref{sec:FourierConstruction}.
Also define auxiliary functions $\mathcal{G}_{\z}:=\xi_0(\bgg_{\z})$, $\jj_n:=\xi_2(\JJ_n)$, and $\mathcal{J}_n:=-\xi_0(\jj_n)$ so that \eqref{eqn:GJxiops} is implied by
\begin{align}
\label{eqn:HGxi}\xi_{2}\!\left(\GG_{\z}\right) &= \bgg_{\z},& \xi_{0}\!\left(\bgg_{\z}\right) &= \mathcal{G}_{\z},&\xi_{2}\!\left(\mathcal{G}_{\z}\right)&=-6,\\
\label{eqn:Jxi}\xi_{2}\!\left(\JJ_n\right)&=\jj_{n},&\xi_{0}\!\left(\jj_{n}\right)&=-\calJ_n,&\xi_{2}\!\left(\calJ_n\right)&=j_n.
\end{align}

Note that throughout the paper, we are using uppercase blackboard for depth four, lowercase blackboard for depth three, uppercase script for depth two, and standard letters for depth one. The exception to this rule is $\widehat{E}_2$, whose notation is standard;
the analogous notation for the Eisenstein series may be found by comparing the functions in \eqref{eqn:diagram} vertically.

\begin{lemma}\label{lem:GJxi}
The functions $\GG_{\z}$ and $\JJ_n$ satisfy \eqref{eqn:GJxiops} and we have
\begin{equation}\label{eqn:xiboldE}
\xi_{2}(\mathbb{E})=\mathcal{E}.
\end{equation}
\end{lemma}
\begin{proof}
Lemma \ref{lem:Laurent} (4) with $f(z,s)=E(z,s)$ and \eqref{defEz} yields that
\[
4 \pi \xi_{2}\!\left(\left[\frac{\partial}{\partial s} \xi_0(E(z,\overline{s}))\right]_{s=1}\right)=\mathcal{E}(z)-\mathcal{C}+12.
\]
Combining this with Lemma \ref{lem:Eprop} (3) yields \eqref{eqn:xiboldE}.

Plugging $f(z,s)=G_s(z,\z)$ into Lemma \ref{lem:Laurent} (4) and using Lemma \ref{lem:Greensprop} (1), we see that
\begin{align}\notag
\xi_2\!\left(\GG_{\z,1}(z)\right) &=  \operatorname{CT}_{s=1}\!\left(G_{s}(z,\z)\right)+ \operatorname{Res}_{s=1} (G_s(z, \z))=\lim_{s\to 1}\!\left(G_s(z,\z) +\tfrac{12}{s-1}\right) -12\\
&= 2\bgg_{\z}(z)-\mathcal{E}(\z)+\mathcal{C}+12,\label{eqn:ggzrewrite}
\end{align}
by \eqref{defEz} and \eqref{eqn:Gzdef}.
%The claim for $\GG_{\z}$ in \eqref{eqn:GJxiops} follows by Lemma \ref{lem:Greensprop} (3).

To show the identity for $\JJ_{n}$ in \eqref{eqn:GJxiops}, first note that if $s \mapsto f(z,s)$ is holomorphic at $s=1$ and $z\mapsto f(z,s)$ is real-differentiable, then
\begin{align}\label{xs}
\left[\xi_\kappa\!\left(f(z,\overline{s})\right)\right]_{s=1}=\xi_\kappa\!\left(f(z,1)\right).
\end{align}
Using this and then interchanging the $\xi$-operator with differentiation in $s$ and recalling that $F_{-n}(z,s)$ is an eigenfunction under $\Delta_0$ with eigenvalue $s(1-s)$, we conclude that
\begin{equation}\label{eqn:xiJJn}
\xi_{2}\!\left(\left[\frac{\partial^2}{\partial s^2} \xi_0\!\left(F_{-n}(z,\overline{s})\right)\right]_{s=1}\right) = 2F_{-n}(z,1)+2\left[\frac{\partial}{\partial s}F_{-n}(z,s)\right]_{s=1}.
\end{equation}
Applying Lemma \ref{lem:Laurent} (4) with  $f(z,s)=F_{-n}(z,s) $, we furthermore have
\begin{equation}\label{eqn:xiFn1}
\xi_2\!\left(\left[\frac{\partial}{\partial s} \xi_0\!\left(F_{-n}(z,\overline{s})\right)\right]_{s=1}\right)=F_{-n}(z,1).
\end{equation}
Combining \eqref{eqn:xiJJn} and \eqref{eqn:xiFn1}
with the definition \eqref{eqn:mathcalFdef}
then gives
\begin{equation}\label{eqn:calFxi}
\xi_{2}\!\left(\mathcal{F}_{-n}(z,1)\right)=2\left[\frac{\partial}{\partial s}F_{-n}(z,s)\right]_{s=1}.
\end{equation}
Applying $\xi_0$ to \eqref{eqn:calFxi} and pulling the $\xi$-operator inside, we conclude that
\begin{equation}\label{eqn:DeltaJJn}
\xi_0\circ\xi_2\!\left(\mathcal{F}_{-n}(z,1)\right)=2\left[\frac{\partial}{\partial s}\xi_0\!\left( F_{-n}(z,\overline{s})\right)\right]_{s=1}.
\end{equation}
Applying $\xi_2$ to \eqref{eqn:DeltaJJn} we then obtain from \eqref{eqn:xiFn1} and  \eqref{eqn:Fnjn} that
\begin{equation*}
\xi_{2}\circ \xi_0\circ\xi_2\!\left(\mathcal{F}_{-n}(z,1)\right)=2F_{-n}(z,1)= \left(\pi\sqrt{n}\right)^{-1}\!\left(j_n(z)+ \mathcal{C}_n\right).
\end{equation*}
The claim then follows, using \eqref{eqn:xiboldE}, Lemma \ref{lem:Eprop} (4), and Lemma \ref{lem:Eprop} (3).
\end{proof}
\begin{remark}
Combining \eqref{eqn:HGxi} and \eqref{eqn:Jxi} together with \eqref{eqn:xiboldE}, Lemma \ref{lem:Eprop} (3), and Lemma \ref{lem:Eprop} (4) yields the following:
\begin{equation}\label{eqn:diagram}\begin{split}
\xymatrix{
\GG_{\z}\ar[r]^{\xi_{2}}&\bgg_{\z}\ar[r]^{\xi_{0}}&\mathcal{G}_{\z}\ar[r]^{\xi_{2}}&-6 ,\\
\JJ_n\ar[r]^{\xi_{2}}&\jj_{n}\ar[r]^{\xi_{0}}&-\calJ_{n}\ar[r]^{\xi_{2}}&-j_n
,\\
\mathbb{E}\ar[r]^{\xi_{2}}&\mathcal{E}\ar[r]^{\xi_{0}}&4\pi \widehat{E}_2\ar[r]^{\xi_{2}}&12.
}
\end{split}
\end{equation}
\end{remark}
\rm

In order to determine the principal parts of the polyharmonic Maass forms defined in \eqref{eqn:GGdef} and \eqref{eqn:scrJndef}, we require the Taylor expansions of $\varphi_{m,s}$ and $\xi_0(g_s(z,\z))$ around $s=1$; note that the principal parts of the analytic continuations to $s=1$ come from the values at $s=1$ of the corresponding seeds of the Poincar\'e series. We compute the first two coefficients of the Taylor expansion of the seeds in the following lemma.
\begin{lemma}\label{lem:diffseeds}
Assume that $n\in\N$.
\noindent

\noindent
\begin{enumerate}[leftmargin=*, label={\rm(\arabic*)}]
\item
We have
\begin{align*}
\varphi_{-n,s}(z)=f_{-n,0}(z) + f_{-n,1}(z)(s-1)+O\!\left((s-1)^2\right),
\end{align*}
where
\begin{align*}
f_{-n,0}(z)&:=\frac{1}{2\pi\sqrt{n}}\left(q^{-n}-W_0(-2\pi ny)q^{-n}\right),\\
f_{-n,1}(z)&:=-\frac{1}{2\pi\sqrt{n}}\left(2\bm{\W}_0(-2\pi ny) q^{-n}+E_1(4\pi ny)q^{-n} \right).
\end{align*}
\item
We have
\begin{equation}\label{eqn:xigsLaurent}
\xi_0\!\left(g_s(z,\z)\right)=\mathfrak{g}_{\z,0}(z) + \mathfrak{g}_{\z,1}(z)(s-1)+O\!\left((s-1)^2\right),
\end{equation}
where
\begin{align*}
\mathfrak{g}_{\z,0}(z)&:=\xi_0\!\left(g_{1}(z,\z)\right),& &\\ \mathfrak{g}_{\z,1}(z)&:=\left(\frac{z-\overline{\z}}{2\sqrt{\y}}\right)^{-2}B(r_{\z}(z)) X_{\z}^{-1}(z),
&\textnormal{with }B(r)&:= \tfrac{4r^2}{\left(1-r^2\right)^2} \left[\frac{\partial}{\partial s} \frac{\partial}{\partial w}Q_{s-1}(w)\right]_{\substack{\hspace{-12pt}s=1,\\w=\tfrac{1+r^2}{1-r^2}}}.
\end{align*}
Moreover, we have
\begin{equation}\label{eqn:rlim}
\lim_{r\to 0^+} B(r)=0.
\end{equation}
\end{enumerate}
\end{lemma}

\begin{proof}
(1) Since $s \mapsto \varphi_{-n,s}(z)$ is holomorphic at $s=1$, we have a Taylor expansion of the shape \eqref{eqn:mathcalFdef}; we next explicitly determine the Taylor coefficients. We obtain $f_{-n,0}(z)= (2\pi\sqrt{n})^{-1}(q^{-n}-\overline{q}^{n})$. Evaluating $I_{\frac12}(w) = \frac{1}{\sqrt{2\pi w}}(e^w-e^{-w})$, the claim for $f_{-n,0}$ then follows by noting that $\Gamma(1,w)=e^{-w}$ for $w>0$ and using Lemma \ref{lem:relationsofnonhol} to evaluate $\overline{q}^{n}=W_0(-2\pi ny)q^{-n}$.

To determine $f_{-n,1}$, we observe that by definition
\[
f_{-n,1}(z)=e^{-2\pi i nx} y^{\frac{1}{2}} \left[\frac{\partial}{\partial s}I_{s-\frac12}(2\pi n y)\right]_{s=1}.
\]
Using \cite[10.38.6]{NIST}, we obtain that
\begin{equation*}
\left[\frac{\partial}{\partial s}I_{s-\frac12}(w)\right]_{s=1} = -(2\pi w)^{-\frac{1}{2}} \left(E_1(2w) e^w + \Ei(2w)e^{-w}\right).
\end{equation*}
Hence, plugging in $w=2\pi ny$, we obtain
\[
f_{-n,1}(z)=-\left(2\pi\sqrt{n}\right)^{-1} e^{-2\pi i nx}\left(E_1(4\pi ny) e^{2\pi ny} + \Ei(4\pi ny)e^{-2\pi ny}\right).
\]
Using Lemma \ref{lem:relationsofnonhol}, one sees that for $w>0$
\begin{align*}
\Ei(w) &= W_2\left(\tfrac{w}{2}\right)+w^{-1}e^w.
\end{align*}
Applying integration by parts to the definition of $\bm{\W}_0$, this implies
\[
\Ei(4\pi n y)e^{-4\pi ny}q^{-n}=2\bm{\W}_0(-2\pi n y)q^{-n},
\]
from which we conclude the claim.

\noindent (2) Define
\begin{equation}\label{eqn:frakgdef}
\mathfrak{g}_s(z,\z):=\frac{\partial}{\partial s}\xi_0\!\left(g_{s}(z,\z)\right).
\end{equation}
By Lemma \ref{lem:Greensprop} (1), the $(s-1)^{-1}$ term in the Laurent
expansion of $g_s$ is
constant as a function of $z$, and hence
annihilated by $\xi_0$. Thus
\eqref{eqn:xigsLaurent} is equivalent to showing that $\mathfrak{g}_1(z,\z)=\mathfrak{g}_{\z,1}(z)$. The chain rule yields
\begin{equation}\label{eqn:frakg}
\mathfrak{g}_1(z,\z)=-2\left[\frac{\partial}{\partial s} \frac{\partial}{\partial w}Q_{s-1}(w)\right]_{s=1,\, w=1+\frac{|z-\z|^2}{2y\y}} \xi_{0}\left(1+\frac{|z-\z|^2}{2y\y}\right).
\end{equation}
A direct computation gives
\begin{align}\label{eqn:xicoshr}
\xi_0\!\left(1+\tfrac{|z-\z|^2}{2y\y}\right)=-2\left(\tfrac{z-\overline{\z}}{2\sqrt{\y}}\right)^{-2} \frac{r_{\z}^2(z)}{\left(1-r_{\z}^2(z)\right)^2} X_{\z}^{-1}(z).
\end{align}
Since
\begin{equation}\label{eqn:coshr}
1+\tfrac{|z-\z|^2}{2y\y}=\cosh(d(z,\z)) = \tfrac{2}{1-r_{\z}^2(z)}-1 = \tfrac{1+r_{\z}^2(z)}{1-r_{\z}^2(z)},
\end{equation}
we conclude that $\mathfrak{g}_1(z,\z)=\mathfrak{g}_{\z,1}(z)$, establishing \eqref{eqn:xigsLaurent}.

To evaluate the limit \eqref{eqn:rlim}, we use (see \cite[Section II, (2.5)]{GrossZagier})
\[
Q_{s-1}(w)=\int_{0}^{\infty}\left(w+\sqrt{w^2-1}\cosh(u)\right)^{-s} du.
\]
It is then not hard to compute
\begin{multline}\label{Qdiff}
%\frac{4r^2}{\left(1-r^2\right)^2} \left[\frac{\partial}{\partial s} \frac{\partial}{\partial w}Q_{s-1}(w)\right]_{\substack{\hspace{-15pt} s=1 \\  w=\frac{1+r^2}{1-r^2}}}\\
B(r)=-4 r\int_0^\infty \frac{1+\log\left(1-r^2\right)-\log\left(1+r^2+2r\cosh(u)\right)}{\left(1+re^u\right)^2\left(1+re^{-u}\right)^2} \left(r+\frac{\left(1+r^2\right)}{2}\cosh(u)\right)du.
\end{multline}
We next determine the limit of this expression as $r\to 0^+$. By evaluating
\[
\int_{0}^{\infty} \frac{r+\frac{1+r^2}{2}\cosh(u)}{\left(1+re^u\right)^2\left(1+re^{-u}\right)^2} du = \frac{1+4r^2\log(r) - r^4}{4\left(1-r^2\right)^3}+O(1),
\]
one can show that the limit of \eqref{Qdiff} as $r\to 0^+$ equals
\begin{equation*}
\lim_{r\to 0^+} \left( r \int_{0}^\infty \frac{\log\left(1+r e^u\right)}{\left(1+re^u\right)^2} e^u du\right)-1.
\end{equation*}
The claim then follows by  determining that the limit equals $1$.
\end{proof}

\section{Fourier expansions}\label{sec:FourierElliptic}

\indent In this section, we investigate the shape of the Fourier expansions of biharmonic Maass forms.

\subsection{Fourier expansions of sesquiharmonic Maass forms}

The following shapes of Fourier expansions for sesquiharmonic Maass forms follow by Lemmas \ref{lem:xiFourier} and \ref{lem:Wgrowth}.
\begin{lemma}\label{lem:sesquiFourier}
If $\mathcal{M}$ is translation-invariant, sesquiharmonic of weight $\kappa\in\Z\setminus\{1\}$, and grows at most linear exponentially at $i\infty$, then for $y\gg0$ we have $\mathcal{M}=\mathcal{M}^{++}+\mathcal{M}^{+-}+\mathcal{M}^{--}$, where
\begin{align*}
\mathcal{M}^{++}(z)&:=\sum_{m\gg -\infty} c_{\mathcal{M}}^{++}(m) q^m,\\
\mathcal{M}^{+-}(z)&:=c_{\mathcal{M}}^{+-}(0) y^{1-\kappa}+\sum_{\substack{m\ll \infty\\ m\neq 0}}c_{\mathcal{M}}^{+-}(m)
W_\kappa(2\pi my)q^m,\\
\mathcal{M}^{--}(z)&:=c_{\mathcal{M}}^{--}(0)\log(y)+\sum_{\substack{m\gg -\infty\\ m\neq 0}} c_{\mathcal{M}}^{--}(m)\bm{\W}_{\kappa}(2\pi my) q^m.
\end{align*}
Moreover, $\mathcal{M}$ is harmonic if and only if $\mathcal{M}^{--}(z)=0$.
\end{lemma}

\subsection{Fourier expansions of biharmonic Maass forms}
A direct calculation gives the following shape of the constant term of the biharmonic part of the Fourier expansion.
\begin{lemma}\label{lem:Fourierconstant}
The constant term of the Fourier expansion of a weight $\kappa\in\Z\setminus\{1\}$ biharmonic Maass form $F$ has the shape
\begin{equation}\label{eqn:Fourierconstant}
c_{F}^{+++}(0)+ c_{F}^{++-}(0) y^{1-\kappa} + c_{F}^{+--}(0) \log(y)+ c_{F}^{---}(0)y^{1-\kappa} \left(1+(\kappa-1)\log(y)\right).
\end{equation}
Moreover, we have
\[
\xi_{\kappa}\!\left(y^{1-\kappa} (1+(\kappa-1)\log(y))\right)=-(\kappa-1)^2\log(y).
\]
\end{lemma}

\subsection{Fourier expansions of the functions from Section \ref{sec:construction}}\label{sec:FourierConstruction}

We now determine the shapes of the Fourier expansions of the functions from Section \ref{sec:construction}. For this, we complete the definition \eqref{eqn:scrJndef} by fixing $a_n$. Specifically, since Lemma \ref{lem:Eprop} (3) implies that $\xi_{2}(a_n \widehat{E}_2)=\frac{3}{\pi}\overline{a_n}$ is a constant, we may choose $a_n$ so that the constant term of the holomorphic part of the Fourier expansion of $\jj_n$ vanishes for $n\neq 0$ and the constant term is $1$ for $n=0$. For  $n=0$ we must verify that the holomorphic part of the constant term is indeed equal to $1$ in the explicit formula $\jj_0(z)=\frac{1}{6}\log(y^6|\Delta(z)|)+1$. For this, we use the product expansion of $\Delta$ to show that as $y\to\infty$
\begin{equation}\label{eqn:logDeltaFourier}
\frac{1}{6}\log\left(y^6|\Delta(z)|\right)+1=
1- \frac{\pi}{3}y+\log(y)+o(1).
\end{equation}

\begin{lemma}\label{lem:Fourierexps}
\noindent

\noindent
\begin{enumerate}[leftmargin=*, label={\rm(\arabic*)}]
\item

 There exist $c_{j_n}(m), c_{\calJ_n}^{+}(m), c_{\calJ_n}^{-}(m), c_{\jj_n}^{++} (m), c_{\jj_n}^{+-}(m)$, and  $c_{\jj_n}^{--}(m)\in\C$ such that
\begin{align*}
%\label{eqn:jnFourier}
j_n(z)&=q^{-n}+\sum_{m\geq 1} c_{j_n}(m) q^m, \\
%\label{eqn:J2Fourier}
\calJ_{n}(z)&=\sum_{m\geq 0} c_{\calJ_{n}}^+(m)q^m +4\pi n
\delta_{n\neq 0}
  W_2(2\pi ny)q^n
-\delta_{n=0} \frac{1}{y}+ \sum_{m\leq-1} c_{\calJ_{n}}^-(m)W_2(2\pi my)q^m,\\
%\label{rewriteJn}
\jj_n(z)&=\delta_{n=0}+\sum_{m\geq 1} c_{\jj_n}^{++}(m)q^m +  c_{\jj_n}^{+-}(0) y+ \sum_{m\leq -1} c_{\jj_n}^{+-}(m)W_{0}(2\pi my) q^m
\notag\\
 &\hspace{2.15cm}
 +\delta_{n=0}\log(y) +2\delta_{n\neq 0}\bm{\W}_{0}(-2\pi ny) q^{-n} +\sum_{m \geq 1} c_{\jj_n}^{--}(m)\bm{\W}_{0}(2\pi my)
 q^m.
\end{align*}
Here $\delta_S:=1$ if some statement $S$ is true and $0$ otherwise.
\item
There exist constants $c_{\bgg_{\z}}^{++}(m)$, $c_{\bgg_{\z}}^{+-}(m)$, $c_{\GG_{\z}}^{+++}(m)$, $c_{\GG_{\z}}^{++-}(m)$, and $c_{\GG_{\z}}^{+--}(m)\in\C$ such that for $y$ sufficiently large
\begin{align*}
%\label{eqn:GcalFourier}
\mathcal{G}_{\z}(z)&= - 4 \pi \sum_{m \geq 1 } m\overline{c^{+-}_{\bgg_{\z}}(-m)} q^m + \frac{6}{y}, \\
%\label{eqn:G1Fourier}
\bgg_{\z}(z) &= \sum_{m\geq 1} c_{\bgg_{\z}}^{++}(m) q^m +\sum_{m \leq -1} c_{\bgg_{\z}}^{+-}(m)W_0(2\pi my)q^m+6\log(y),\\
%\label{eqn:GboldFourier}
\GG_{\z}(z)&=\sum_{m\geq 0} c_{\GG_{\z}}^{+++}(m)q^m + \sum_{m\leq  -1} c_{\GG_{\z}}^{++-}(m)
W_{2}(2\pi my) q^m \notag\\
& \hspace{4.0cm}-\frac{6}{y}\left(1+\log(y)\right)+\sum_{m \geq 1}c_{\GG_{\z}}^{+--}(m)\bm{\W}_{2}(2\pi my)q^m.
\end{align*}
\item
There exist constants $c_{\mathbb{E}}^{+++} (m)$, $c_{\mathbb{E}}^{++-}(m)$, $c_{\mathbb{E}}^{+--}(m)$, $c_{\mathcal{E}}^{++}(m)$, and $c_{\mathcal{E}}^{+-}(m)\in\C$ such that
\begin{align*}
%\label{xiE}
\xi_0(\mathcal E(z))&=4\pi \widehat{E}_2(z) = 4\pi \sum_{m\geq 0} c_{E_2} (m) q^m -\frac{12}{y},\\
%\label{EFourier}
\mathcal E(z)& = \sum_{m\geq 1 } c_{\mathcal E}^{++}(m) q^m+ 4\pi y +\sum_{m\leq -1}c_{\mathcal E}^{+-}(m) W_{0}(2\pi my) q^m- 12\log(y),\\
\nonumber \mathbb E(z) &= \sum_{m\geq 0} c_{\mathbb E}^{+++} (m) q^m + \sum_{m\leq -1} c_{\mathbb E}^{++-}(m) W_2(2\pi my)q^m \\%+ \frac{c_{\mathbb E}^{++-}(0)}{y}\\
&\hspace{3.2cm}+ 4\pi  \log(y) +\sum_{m\geq 1} c_{\mathbb E}^{+--}(m)\bm{\W}_{2}(2\pi my)q^m+ \frac{12}{y}(1+\log(y)).
%\label{eqn:boldEFourier}
\end{align*}

\end{enumerate}
\end{lemma}
\begin{proof}
\noindent
(1) Since the expansion for $j_n$ is well-known, it is enough to show the expansion for $\jj_n$. The expansion for $\calJ_n$ then follows by applying $\xi_0$, employing \eqref{eqn:Jxi} and Lemma \ref{lem:xiFourier}. We now use \eqref{eqn:calFxi}, then apply $\frac{\partial}{\partial s}$ to the Fourier expansion of $F_{-n}(z,s)$ given in \cite[Theorem 1]{Niebur}, and employ Lemma \ref{lem:diffseeds} (1) to determine the contribution to the principal part from the first term in \eqref{eqn:scrJndef}. Combining this with \eqref{defEz} and \eqref{eqn:xiboldE}, we see that the principal part of $\jj_n$ is the growing part of
\begin{equation}\label{eqn:growthjjn}
c_{\jj_n}^{+-}(0) y+\delta_{n=0}\log(y)+\delta_{n\neq 0}\left(2\bm{\W}_0(-2\pi ny)q^{-n}+E_1(4\pi ny)q^{-n}\right).
\end{equation}
However, by \eqref{eqn:GammaEr} and the asymptotic growth of the incomplete gamma function \cite[8.11.2]{NIST}, $E_1(4\pi ny)q^{-n}$ decays exponentially as $y\to\infty$ and thus does not contribute to the principal part. The constant term of the holomorphic part of $\jj_n$ is determined by the choice of $a_n$. In the special case $n=0$, the evaluation $c_{\jj_0}^{+-}(0)=-\frac{\pi}{3}$
implies that \eqref{eqn:growthjjn} matches \eqref{eqn:logDeltaFourier}.

\noindent
(2) We first claim that $\bgg_{\z}(z)-6\log(y)$ vanishes as $y\to\infty$. By \cite[Section II, (2.19)]{GrossZagier} we have
\begin{equation}\label{eqn:Gs*lim}
G_{s}(z,\z) = \frac{4\pi}{1-2s}y^{1-s} E(\z,s)+O_s\!\left(e^{-y}\right) \quad (\text{as}\ y \to \infty),
\end{equation}
where the error has no pole at $s=1$.
Combining \eqref{eqn:Gs*lim} with \eqref{eqn:ggzrewrite} and \eqref{defEz} implies that
\begin{equation}\label{eqn:aG1++(0)}
\bgg_{\z}(z)=\lim_{s\to 1} \left( 2\pi \left(\frac{y^{1-s}}{1-2s}+1\right)E(\z,s)\right)-12
+O\!\left(e^{-y}\right)\qquad\textnormal{as }y\to\infty.
\end{equation}
From \cite[p. 241, second displayed formula]{GrossZagier}, we have
\[
2\pi E(\z,s)=\frac{6}{s-1}+ O(1), \qquad 1+\frac{y^{1-s}}{1-2s}=\left(\log(y)+2\right)(s-1)+O_y\left((s-1)^2\right),
\]
and hence the right-hand side of \eqref{eqn:aG1++(0)} equals $6\log(y)+O(e^{-y})$. We conclude from \eqref{eqn:aG1++(0)} that $\bgg_{\z}(z)-6\log(y)$ vanishes as $y\to\infty$.

Since the sesquiharmonic part of $\bgg_{\z}$ is $6\log(y)$ by Lemma \ref{lem:Greensprop} (3), the expansion for $\bgg_{\z}$ follows from the expansion in Lemma \ref{lem:sesquiFourier} by comparing the asymptotics in Lemma \ref{lem:Wgrowth} with \eqref{eqn:aG1++(0)}. The expansion for $\mathcal{G}_{\z}$ then follows by applying $\xi_0$ and using Lemma \ref{lem:xiFourier}.

Finally, by explicitly computing a pre-image under $\xi_2$ of the sesquiharmonic part of  $\bgg_{\z}$, we conclude that the biharmonic part of the expansion of $\GG_{\z}(z)$ is $-6(1+\log(y))y^{-1}$. Subtracting this from $\GG_{\z}$ yields a sesquiharmonic function which is bounded as $y\to \infty$ by \eqref{eqn:Gs*lim} and \eqref{eqn:GGdef}. We may thus use Lemma \ref{lem:sesquiFourier} and note the asymptotics in Lemma \ref{lem:Wgrowth} to compute the shape of the rest of the expansion.

\noindent

\noindent

\noindent
(3) The fact that $\xi_0(\mathcal E)=4\pi\widehat{E}_2$ follows from Lemma \ref{lem:Eprop} (4). Using Lemma \ref{lem:Wgrowth}, we obtain the claim for $\mathcal{E}$ directly from \eqref{eqn:calEgrowth}. Noting the relationship \eqref{eqn:xiboldE} between $\mathbb{E}$ and $\mathcal{E}$ together with $\xi_2(y^{-1})=-1$ and $\xi_2(\log(y))=y$, the constant term of $\mathbb{E}$ is then obtained by Lemma \ref{lem:Fourierconstant}. After subtracting this term, the remaining function is sesquiharmonic and we obtain the Fourier expansion of $\mathbb{E}$ by Lemma \ref{lem:sesquiFourier} and Lemma \ref{lem:Wgrowth}.
\end{proof}

\section{Elliptic expansions}
\subsection{Elliptic expansions of polyharmonic Maass forms}
For a weight $\kappa$ non-holomorphic modular form $F$, the \begin{it}elliptic expansion\end{it} around the point $\z\in\H$ is the unique expansion of the type
\[
F(z)=\left(\frac{z-\overline{\z}}{2\sqrt{\y}}\right)^{-\kappa}\sum_{m\in\Z} c_{F,\z}\!\left(r_{\z}(z);m\right) X_{\z}^m(z), \qquad r_\z(z) \ll 1,
\]
where  $r_{\z}$ and $X_{\z}$ are defined in \eqref{eqn:rXdef}.
\rm
If $F$ is a polar polyharmonic Maass form of depth $\ell$, then $c_{F,\z}(r;m)$ satisfies a differential equation with $\ell$ independent solutions
for each $m\in\Z$.
We choose a basis of solutions $B_{\kappa,j}(r;m)$ for $1\leq j\leq \ell$ such that $(z-\overline{\z})^{-\kappa}B_{\kappa,j}(r_{\z}(z);m)X_{\z}^{m}(z)$ has depth $j$, i.e., $j\in\N_0$ is minimal with
\begin{equation*}
\xi_{\kappa,z}^{j}\!\left(
\left(\frac{z-\overline{\z}}{2\sqrt{\y}}\right)^{-\kappa}
B_{\kappa,j}\!\left(r_{\z}(z);m\right)X_{\z}^{m}(z)\right)=0.
\end{equation*}
We then write, iterating ``$+$'' $\ell-j$ times and ``$-$'' $j-1$ times,
\[
c_{F,\z}\!\left(r_{\z}(z);m\right) = \sum_{j=1}^{\ell}c_{F,\z}^{+\ldots+-\ldots-}(m) B_{\kappa,j}\!\left(r_{\z}(z);m\right).
\]

A direct calculation gives the following lemma.
\begin{lemma}\label{lem:xiellexp}
Let
 $g_m:(0,1) \to \R$ be a differentiable function and define
\[
f_m(z):=\left(\frac{z-\overline{\z}}{2\sqrt{\y}}\right)^{-\kappa} g_m(r_{\z}(z)) X_{\z}^m(z).
\]
Then the following hold.
\noindent

\noindent
\begin{enumerate}[leftmargin=*, label={\rm(\arabic*)}]
\item
For $\kappa\in \Z$, we
 have
\[
\xi_{\kappa}\!\left(f_m(z)\right)=  -\frac{1}{2}\left(\frac{z-\overline{\z}}{2\sqrt{\y}}\right)^{\kappa-2} \left(1-r_{\z}^2(z)\right)^{\kappa}r_{\z}^{2m+1}(z)  g_{m}'\!\left(r_{\z}(z)\right)X_{\z}^{-m-1} (z),
\]
\item
The function $f_m$ has depth $j$ if and only if there exist $a_{\ell}\, (1 \leq \ell \leq j-1)$ with $a_{j-1}\neq 0$ for which
\[
g_m'(r) =\left(1-r^2\right)^{-\kappa}r^{-2m-1} \sum_{\ell=1}^{j-1} a_{\ell} B_{2-\kappa,\ell}(r;-m-1).
\]
\end{enumerate}
\end{lemma}

By Lemma \ref{lem:xiellexp} (2), one may choose $B_{\kappa,j}(r;m)$ such that, for some $
C_m
\neq 0$,
\begin{equation}\label{eqn:Bxirel}
B_{\kappa,j}'(r;m)= C_m \left(1-r^2\right)^{-\kappa}r^{-2m-1} B_{2-\kappa,j-1}(r;-m-1).
\end{equation}
This uniquely determines $r\mapsto B_{\kappa,j}(r;m)$ up to an additive constant. If $B_{\kappa,j}(r;m)$ satisfies \eqref{eqn:Bxirel} and $f$ has depth $j$, then we say that $f$ has \begin{it}pure depth\end{it} $j$ if there exist constants $c_{\z}(m)\in\C$ such that
\[
f(z)=(z-\overline{\z})^{-\kappa}\sum_{m\in \Z} c_{\z}(m) B_{\kappa,j}(r;m) X_{\z}^m(z).
\]

Since any function $f$ of depth $\ell$ naturally (and uniquely) splits into $f=\sum_{j=1}^{\ell} f_j$ with $f_j$ of pure depth $j$, we call $f_j$ the \begin{it}depth $j$ part of the elliptic expansion of $f$.
\end{it}

\subsection{General elliptic expansions of sesquiharmonic Maass forms}
We next explicitly choose a basis of functions $B_{\kappa,j}(r;m)$ in the special case that $F$ is sesquiharmonic. For this we define
\begin{align}\label{eqn:B23def}
B_{\kappa,2}(r;m)&:=\beta_{t_0}(1-r^2;1-\kappa,-m),& B_{\kappa,3}(r;m)&:=\bm{\beta}_{\kappa-1,-m}(r),\qquad 0<t_0,r<1
\end{align}
with
\begin{align*}
\beta_{t_0}(r;a,b)&:=-\int_{r}^{1-t_0} t^{a-1}(1-t)^{b-1}dt-\sum_{\substack{n\geq 0\\ n\neq -b}} \frac{(-1)^n}{n+b}\binom{a-1}{n} t_0^{n+b}-(-1)^b  \delta_{a\in\N}\delta_{0\leq -b<a}\log\!\left(t_0\right),\\
\bm{\beta}_{a,b}(r)&:=-2\int_0^{r}t^{2b-1}\left(1-t^2\right)^{-a-1}\beta_{t_0}\!\left(1-t^2; a, 1-b\right)dt.
\end{align*}
It turns out that $\beta_{t_0}$ is independent of the choice of $t_0$ (see Lemma \ref{lem:betabnd} below). We however leave the dependence in the notation to distinguish it from the incomplete beta function.
\begin{definition}\label{def:finorder}
	We say that a function $\mathcal{M}$ has \begin{it}finite order\end{it} at $\z$ if there exists $m_0\in\R$ such that $r_{\z}^{m_0}(z)\mathcal{M}(z)$ does not have a singularity at $\z$.
\end{definition}

\begin{lemma}\label{lem:sesquiellexp}
If $\mathcal{M}$ is sesquiharmonic of weight $\kappa$ and has singularities of finite order, then it
 has an expansion of the type $\mathcal{M}=\mathcal{M}_{\z}^{++} + \mathcal{M}_{\z}^{+-}+\mathcal{M}_{\z}^{--}$ for $r_{\z}(z)$ sufficiently small with
	\begin{align*}
	\mathcal{M}_{\z}^{++}(z)&:=\left(\frac{z-\overline{\z}}{2\sqrt{\y}}\right)^{-\kappa}\sum_{m\gg -\infty} c_{\mathcal{M},\z}^{++}(m)X_{\z}^m(z),\\
	\mathcal{M}_{\z}^{+-}(z)&:=\left(\frac{z-\overline{\z}}{2\sqrt{\y}}\right)^{-\kappa}\sum_{m\ll \infty} c_{\mathcal{M},\z}^{+-}(m)\beta_{t_0}\!\left(1-r_{\z}^2(z);1-\kappa,-m\!\right)X_{\z}^m(z),\\
	\mathcal{M}_{\z}^{--}(z)&:=\left(\frac{z-\overline{\z}}{2\sqrt{\y}}\right)^{-\kappa}\sum_{m\gg -\infty} c_{\mathcal{M},\z}^{--}(m)\bm{\beta}_{\kappa-1,-m}\!\left(r_{\z}(z)\right)X_{\z}^m(z)
	\end{align*}
for some constants $c_{\mathcal{M},\z}^{++}(m),c_{\mathcal{M},\z}^{+-}(m)$, and $c_{\mathcal{M},\z}^{--}(m)\in\C$.
Moreover, the special functions in this expansion, i.e, those given in \eqref{eqn:B23def}, satisfy \eqref{eqn:Bxirel}
with $C_m=-2$.
\end{lemma}
\begin{proof}
We use Lemma \ref{lem:xiellexp} with $g_m=g_{m,j}$, where for $j\in\{1,2,3\}$ we define
\begin{align*}
g_{m,1}(r):=1,\qquad g_{m,2}(r):=\beta_{t_0}\left(1-r^2;1-\kappa,-m\right),\qquad g_{m,3}(r):=\bm{\beta}_{\kappa-1,-m}(r).
	\end{align*}
	We show that $\xi_{\kappa}^3$ annihilates the corresponding functions $f_m=f_{m,j}$ and $f_{m,j}$ has depth $j$. Clearly $f_1$ has depth one as it is meromorphic. Computing
\begin{equation}\label{eqn:B2diff}
g_{m,2}'(r)=-2r\left[\frac{\partial}{\partial w} \beta_{t_0}(w;1-\kappa,-m)\right]_{w=1-r^2} = -2 \left(1-r^2\right)^{-\kappa} r^{-2m-1},
\end{equation}
we see from Lemma \ref{lem:xiellexp} (1) that
\begin{equation}\label{eqn:xiB2}
\xi_{\kappa}\!\left(f_{m,2}(z)\right)=\left(\frac{z-\overline{\z}}{2\sqrt{\y}}\right)^{
\kappa-2}X_{\z}^{-m-1}(z)
\end{equation}
is meromorphic, and hence $f_{m,2}$ has depth two.

Finally we compute
\begin{equation}\label{eqn:B3diff}
\bm{\beta}_{\kappa-1,-m}'(r)= -2 r^{-2m-1} \left(1-r^2\right)^{-\kappa}\beta_{t_0}\left(1-r^2;\kappa-1,m+1\right),
\end{equation}
from which we conclude via Lemma \ref{lem:xiellexp} (1) that
\begin{equation}\label{eqn:xiB3}
\xi_{\kappa}\!\left(f_{m,3}(z)\right)=
\left(\frac{z-\overline{\z}}{2\sqrt{\y}}\right)^{\kappa-2}
\beta_{t_0}\left(1-r_{\z}^2(z);\kappa-1,m+1\right)X_{\z}^{-m-1}(z).
\end{equation}
Thus $f_{m,3}$ has depth three from the above calculation.
\end{proof}

The following lemma follows directly by using the binomial series.
\begin{lemma}\label{lem:betabnd}
	For $0<w<1$ and $0<t_0<1$
	we have
\[%\begin{equation}\label{rewriteb}
	\beta_{t_0}(w;a,b)=- \sum_{n\in\N_0\setminus\{-b\}} \frac{(-1)^n}{n+b}\binom{a-1}{n}(1-w)^{n+b}-(-1)^b \delta_{a\in\N}\delta_{0\leq -b<a}\log(1-w).
\]
	In particular, $\beta_{t_0}$ is independent of $t_0$ and as $w\to 1^-$ we have
\[%	\begin{equation}\label{growth}
\beta_{t_0}(w;a,b)\sim \begin{cases}
-\frac{1}{b}(1-w)^{b} & \text{if }b\neq 0,\\
(a-1) \delta_{a\notin\N}(1-w) -\log(1-w)\delta_{a\in\N}&\text{if }b=0.\end{cases}
\]%	\end{equation}
\end{lemma}
Lemma \ref{lem:betabnd} then directly implies the following corollary.
\begin{corollary}\label{cor:boldbetabnd}
As $r\to 0^+$, we have
\[
\bm{\beta}_{\kappa-1,-m}(r)\sim
\begin{cases}-\frac{1}{m+1}r^2&\text{if }m\neq -1,\\
\frac{2-\kappa}{2} r^4\delta_{\kappa<2} + 2 r^2\log(r) \delta_{\kappa\geq 2} &\text{if }m=-1.
\end{cases}
\]
\end{corollary}

Using Lemma \ref{lem:betabnd} and Corollary \ref{cor:boldbetabnd}, one may determine those terms in the elliptic expansion that contribute to the principal part and those that do not.
\begin{lemma}\label{lem:sesquiellpp}
Suppose that $\mathcal{M}$ is sesquiharmonic of weight $\kappa$ and has singularities of finite order. Then the following hold:
\begin{enumerate}[leftmargin=*, label={\rm(\arabic*)}]
\item
The principal part of the meromorphic part $\mathcal{M}_{\z}^{++}$ of the elliptic expansion of $\mathcal{M}$ precisely comes from those $m$ with $m<0$.
\item
If $\kappa\in-\N_0$, then the principal part of the harmonic part $\mathcal{M}_{\z}^{+-}$ of the elliptic expansion of $\mathcal{M}$ precisely comes from those $m$ with $m\geq 0$.
\item
If $\kappa\notin-\N_0$, then the principal part of the harmonic part $\mathcal{M}_{\z}^{+-}$ of the elliptic expansion of $\mathcal{M}$ precisely comes from those $m$ with $m>0$.
\item
The principal part of the sesquiharmonic part $\mathcal{M}_{\z}^{--}$ of $\mathcal{M}$ precisely comes from those $m$ with $m\leq -3$.
\end{enumerate}
\end{lemma}

\subsection{Biharmonic parts of elliptic expansions}
For biharmonic forms we only require the coefficient of $X_{\z}^{-1}(z)$ of the biharmonic part of their elliptic expansion. By Lemma \ref{lem:betabnd} and Corollary \ref{cor:boldbetabnd}, $B_{2,2}(r;-1)=\beta_{t_0}(1-r^2;-1,1)$ and $B_{2,3}(r;-1)=\bm{\beta}_{1,1}(r)$ both vanish as $r\to 0^+$. We next show that $B_{2,4}(r;-1)$ may be chosen so that it also decays as $r\to 0^+$. The following lemma follows directly by \eqref{eqn:B3diff} and Lemma \ref{lem:betabnd}.

\begin{lemma}\label{lem:doublyharmonicellexp}
The function
\[
B_{2,4}(r;-1):=\frac{\log\left(1-r^2\right)+r^2}{1-r^2}
\]
satisfies \eqref{eqn:Bxirel} with $C_m=-2$, and hence the corresponding function from Lemma \ref{lem:xiellexp} has depth four. In particular, $\lim_{r\to 0^+} B_{2,4}(r;-1)=0$.
\end{lemma}

\subsection{Elliptic expansions of the functions from Section \ref{sec:construction}} We next explicitly determine the shape of the expansions of the functions defined in Section \ref{sec:construction}.
\begin{lemma}\label{lem:ellexps}
	\noindent
	\begin{enumerate}[leftmargin=*, label={\rm(\arabic*)}]
		\item
 		There exist $c_{j_n,\z}(m)$, $ c_{\calJ_n,\z}^{+}(m)$, $c_{\calJ_n,\z}^-(m)$, $c_{\jj_n,\z}^{++}(m)$, $c_{\jj_n,\z}^{+-}(m)$, and $c_{\jj_n,\z}^{--}(m)\in\C$ such that we have for $r_{\z}(z)$ sufficiently small
		\begin{align*}
%\label{jelliptic}
j_n(z)&=\sum_{m\geq 0} c_{j_n,\z}(m) X_{\z}^m(z),\\
%\label{eqn:J2ellexp}
\calJ_{n}(z)&=\left(\frac{z-\overline{\z}}{2\sqrt{\y}}\right)^{-2}\left(\sum_{m\geq 0} c_{\calJ_{n},\z}^+(m)X_{\z}^m(z) + \sum_{m\leq -1}c_{\calJ_{n},\z}^-(m) \beta_{t_0}\!\left(1-r_{\z}^2(z);-1-m\right)X_{\z}^m(z)\right),\\
\notag \jj_n(z)&=\sum_{m\geq 0} c_{\jj_n,\z}^{++}(m)X_{\z}^m(z) + \sum_{m\leq-1} c_{\jj_n,\z}^{+-}(m)\beta_{t_0}\!\left(1-r_{\z}^2(z);1-m\right)X_{\z}^m(z)\\
%\label{eqn:J3ellexp}
&\hspace{2.8in}		+ \sum_{m\geq 0}c_{\jj_n,\z}^{--}(m)\bm{\beta}_{-1,-m}\!\left(r_{\z}(z)\right)X_{\z}^m(z).
\end{align*}
\item  For every $w\in\H$, there exist $c_{\GG_{\z},w}^{+++}(m)$, $c_{\GG_{\z},w}^{++-}(m)$, $c_{\GG_{\z},w}^{+--}(m)$, $c_{\bgg_{\z},w}^{++}(m)$, $c_{\bgg_{\z},w}^{+-}(m)$, and $c_{\bgg_{\z},w}^{--}(m)\in\C$ such that for $r_{w}(z)$ sufficiently small we have
\begin{align*}
%\label{eqn:xi0G1zellexp}
		\mathcal{G}_{\z}(z)&= \left(\frac{z-\overline{w}}{2\sqrt{\im(w)}}\right)^{-2}\!\!\left(c_{\bgg_{\z},w}^{+-}(0)X_{w}^{-1}(z)+ \sum_{m \geq 0} \overline{c^{+-}_{\bgg_{\z},w}(-m-1)} X_w^{m}(z) + \frac{6 r_{w}^2(z)}{1- r_w^2(z)} X_w^{-1} (z) \right),\\
%\label{eqn:G1ellexp}
		\bgg_{\z}(z) &= \sum_{m\geq 0} c_{\bgg_{\z},w}^{++}(m) X_{w}^m(z) +\sum_{m\leq  0} c_{\bgg_{\z},w}^{+-}(m)\beta_{t_0}\!\left(1-r_{w}^2(z);1,-m\right) X_{w}^m(z)+ 6 \log\!\left(1-r_{w}^2(z)\right),\\
%\label{eqn:boldGellexp}
		\GG_{\z}(z)&=\left(\frac{z-\overline{w}}{2\sqrt{\im(w)}}\right)^{-2}\Bigg(\sum_{m\geq 0} c_{\GG_{\z},w}^{+++}(m)X_{w}^{m}(z) + \sum_{m\leq -1} c_{\GG_{\z},w}^{++-}(m)\beta_{t_0}\!\left(1-r_{w}^2(z);-1,-m\right) X_{w}^m(z)\\
	&\hspace{1.1in}	+\sum_{m\geq -1} c_{\GG_{\z},w}^{+--}(m)\bm{\beta}_{1,-m}\!\left(r_{w}(z)\right) X_{w}^m(z)-6\frac{\log\!\left(1-r_{w}^2(z)\right)+r_w^2(z)}{1-r_{w}^2(z)}X_{w}^{-1}(z)\Bigg).
\end{align*}
Moreover,  $c_{ \bgg_{\z},w}^{+-}(0)=0$  unless $w$ is equivalent to $\z$, in which case we have $c_{\bgg_{\z},\z}^{+-}(0)=-\frac{\omega_{\z}}{2}$.
\end{enumerate}
\end{lemma}

\begin{proof}

\noindent
(1) Lemma \ref{lem:sesquiellpp}
gives the claim, up to the vanishing of the terms $m=-1$ and $m=-2$ in the sesquiharmonic part. Applying \eqref{eqn:xiB3} and then \eqref{eqn:xiB2} to these terms, this vanishing follows from the fact that $j_n=\Delta_0(\jj_n)$ does not have a singularity. The other claims then follow by \eqref{eqn:xiB2} and \eqref{eqn:xiB3}.

\noindent
(2) We begin by proving the expansion for $\GG_{\z}$ and we first determine its biharmonic part. Since $\bgg_{\z}(z)-6\log(y)$ is harmonic by Lemma \ref{lem:Greensprop} (3), $\xi_2$ maps the biharmonic part of $\GG_{\z}$ to $6$ times the sesquiharmonic part of the elliptic expansion of $\log(y)$, which by \eqref{eqn:coshr} may be computed via
\begin{equation}\label{eqn:logy}
\log(y)=\log(\im(w))+\log\left(1-r_{w}^2(z)\right)-\Log\!\left(1-X_{w}(z)\right)-\Log\!\left(1-\overline{X_{w}(z)}\right).
\end{equation}
Since the third and fourth terms are harmonic, they do not contribute to the sesquiharmonic part, and hence the sesquiharmonic part of $\bgg_{\z}$ is precisely $6\log(1-r_{w}^2(z))=-6B_{0,3}(r_{w}(z);0)$. Hence, by Lemma \ref{lem:doublyharmonicellexp} and Lemma \ref{lem:xiellexp} (1), we see that
\begin{equation}\label{eqn:doublyHHhat}
\GG_{\z}(z)+6\left(\frac{z-\overline{w}}{2\sqrt{\im(w)}}\right)^{-2}\frac{\log\!\left(1-r_{w}^2(z)\right)+r_w^2(z)}{1-r_{w}^2(z)}X_{w}^{-1}(z)
\end{equation}
is sesquiharmonic. Choosing the function from Lemma \ref{lem:doublyharmonicellexp} as a basis element for the biharmonic part of the elliptic expansion, \eqref{eqn:doublyHHhat} is precisely the
sum of the terms in the elliptic expansion of $\GG_{\z}$ of depth at most three. The shape of this expansion is given in Lemma \ref{lem:sesquiellexp}, and we are left to determine the ranges of the summation. Recall that $\xi_2$ sends the $X_{\z}^{m}(z)$ term to an $X_{\z}^{-m-1}(z)$ term by Lemma \ref{lem:xiellexp} and the principal part of $\xi_{2}(\mathbb{G}_{\z})=\bgg_{\z}$ is a constant multiple of $\log(r_{\z}(z))=-\frac{1}{2}\beta_{t_0}(1-r_{\z}^2(z);1;0)$, which follows directly from Lemma \ref{lem:betabnd}. Thus the sesquiharmonic principal part comes from an $m=-1$ term in the elliptic expansion of $\GG_{\z}$ and it remains to explicitly compute the meromorphic principal part of $\GG_{\z}$. Since for $\re(s)>1$ the principal part of $\GG_{\z,s}$ comes from $2\omega_{\z}\mathfrak{g}_{s}(z,\z)$, taking the limit we conclude that the principal part of $\GG_{\z}$ comes from $\omega_{\z}\mathfrak{g}_{\z,1}$ by Lemma \ref{lem:diffseeds} (2). Since $g_{s}(z,\z)$ is an eigenfunction under $\Delta_0$ with eigenvalue $s(1-s)$, we may use Lemma \ref{lem:Laurent} (2) with $\kappa=0$ to conclude that $\mathfrak{g}_{\z,1}$ is annihilated by $\Delta_2^2$. Since $\mathfrak{g}_{\z,1}(z)=(\frac{z-\overline{\z}}{2\sqrt{\y}})^{-2} B(r_{\z}(z)) X_{\z}^{-1}(z)$ is biharmonic, we may hence write the elliptic expansion of $\mathfrak{g}_{\z,1}$ as
\[
\mathfrak{g}_{\z,1}(z)=\left(\frac{z-\overline{\z}}{2\sqrt{\y}}\right)^{-2}\sum_{j=1}^4 a_jB_{2,j}\!\left(r_{\z}(z);-1\right)  X_{\z}^{-1}(z)
\]
for some $a_j\in\C$, and we see that $c_{\GG_{\z},\z}^{+++}(-1)=\omega_{\z}a_1$.

Moreover,  Lemmas \ref{lem:sesquiellexp} and \ref{lem:doublyharmonicellexp} together with \eqref{eqn:frakg} and \eqref{eqn:xicoshr} imply that
\begin{equation}\label{eqn:constQs-1}
%\frac{4r^2}{\left(1-r^2\right)^2}\left[\frac{\partial}{\partial s} \frac{\partial}{\partial w}Q_{s-1}(w)\right]_{\substack{\hspace{-18pt}s=1 \\ w=\frac{1+r^2}{1-r^2}}}
B(r)= a_1+a_2\beta_{t_0}\!\left(1-r^2; -1,1\right) +a_3\bm{\beta}_{1,1}(r)+a_4\frac{\log\!\left(1-r^2\right)+r^2}{1-r^2}.	
\end{equation}	
We then note that by Lemma \ref{lem:betabnd} and Corollary \ref{cor:boldbetabnd}, the limit $r\to 0^+$ on the right-hand side of \eqref{eqn:constQs-1} equals $a_1$. Using \eqref{eqn:rlim} to evaluate the limit on the left-hand side of \eqref{eqn:constQs-1}, we obtain $a_1=0$. We hence conclude the claim for $\GG_{\z}$ by Lemma \ref{lem:sesquiellpp}.

In order to obtain the elliptic expansion for $\bgg_{\z}$, we apply the operator $\xi_2$ to the expansion for $\GG_{\z}$. Using Lemma \ref{lem:xiellexp} (1) and \eqref{eqn:Bxirel} (together with \eqref{eqn:B2diff} and \eqref{eqn:B3diff}) and Lemma \ref{lem:doublyharmonicellexp}, we conclude the claimed expansion. Similarly, we apply $\xi_0$ to the expansion of $\bgg_{\z}$ and use Lemma \ref{lem:xiellexp} and \eqref{eqn:Bxirel} to obtain the claimed expansion for $\mathcal{G}_{\z}$.
$\qedhere$
	
\end{proof}

\section{Evaluating the inner product and the proof of Theorem \ref{thm:jninnergen}}\label{sec:jninner}

To define the regularization of the inner product used in this paper, we set
\begin{equation*}
\mathcal F_{T,\z_1,\ldots,\z_\ell,\varepsilon_1,\ldots,\varepsilon_\ell} := \mathcal F_T\setminus \bigcup_{j=1}^{\ell}\left(\mathcal B_{\varepsilon_j}(z_j)\cap \mathcal F_T\right),
\end{equation*}
where $\mathcal F_T$ is the usual fundamental domain cut off at height $T$ and
\begin{equation*}
\mathcal{B}_{\varepsilon_j}\!\left(\z_j\right):=\left\{ z\in\H:r_{\z_j}(z)<\varepsilon_j\right\}.
\end{equation*}
For two functions $g$ and $h$ satisfying weight $k$ modularity and whose singularities in the fundamental domain lie in $\{\z_1,\dots \z_{\ell},i\infty\}$, we define the generalized inner product (in case of existence)
\begin{equation}\label{eqn:innerdef}
\left<g,h\right>:=\lim_{T\to\infty}\lim_{\varepsilon_\ell\to 0^+}\cdots\lim_{\varepsilon_{1}\to 0^+}\left<g,h\right>_{T,\varepsilon_1,\dots,\varepsilon_{\ell}},
\end{equation}
where
\[
\left<g,h\right>_{T,\varepsilon_1,\dots,\varepsilon_{\ell}}:=\int_{\mathcal{F}_{T,\mathfrak{z}_1,\dots,\mathfrak{z}_{\ell},\varepsilon_1,\dots,\varepsilon_{\ell}}}
g(z)\overline{h(z)} y^k \frac{dxdy}{y^2}.
\]

Before explicitly computing the inner product with $j_n$, we give the following lemma for evaluating the inner product, which follows by repeated usage of Stokes' Theorem.
\begin{lemma}\label{lem:innerggen}
Suppose that $\mathbb{F}_1,\mathbb{F}_2:\H\to\C$ satisfy weight two modularity and that $\xi_2\circ \Delta_{2}(\mathbb{F}_2)=C$ is a constant. Then, denoting  $\mathbbm{f}_j:=\xi_{2}(\mathbb{F}_j)$, $
\mathcal F_j:=\xi_0(\mathbbm{f}_j)
$, and $f_j:=-\xi_2(\mathcal F_j)$, we have
\begin{multline}\label{eqn:innerggen}
\left<f_1,\mathbbm{f}_2\right>_{T,\varepsilon_1,\dots,\varepsilon_{\ell}} = \overline{\int_{\partial \mathcal{F}_{T,\mathfrak{z}_1,\dots,\mathfrak{z}_{\ell},\varepsilon_1,\dots,\varepsilon_{\ell}}} \mathcal{F}_1(z)\mathbbm{f}_{2}(z) dz}\\
-\int_{\partial \mathcal{F}_{T,\mathfrak{z}_1,\dots,\mathfrak{z}_{\ell},\varepsilon_1,\dots,\varepsilon_{\ell}}}\mathbbm{f}_1(z)\mathcal{F}_{2}(z)dz-\overline{C}\overline{\int_{\partial \mathcal{F}_{T,\mathfrak{z}_1,\dots,\mathfrak{z}_{\ell},\varepsilon_1,\dots,\varepsilon_{\ell}}} \mathbb{F}_1(z)dz}.
\end{multline}
Here the integral along the boundary $\partial \mathcal{F}_{T}$ is oriented counter-clockwise and the integral along the boundary $\partial \mathcal{B}_{\varepsilon_j}(\mathfrak{z}_j)$ is oriented clockwise. In particular $\left<f_1,\mathbbm{f}_2\right>$ equals the limit as $T\to \infty$, $\varepsilon_j\to 0^+$, $1\leq j \leq \ell$, of the right-hand side of \eqref{eqn:innerggen}, assuming that the regularized integrals exist.
\end{lemma}
We are now ready to prove Theorem \ref{thm:jninnergen} with the explicit constant
\begin{equation}\label{definecn}
c_n:= 6\overline{ c_{\JJ_n}^{+++}(0)}+  6 c_{\jj_n}^{+-}(0).
\end{equation}

\begin{proof}[Proof of Theorem \ref{thm:jninnergen}]
We begin by recalling the well-known fact that every meromorphic modular form $f$ may be written as a product of the form (for example, see \cite[(61)]{Pe7})
\[
f(z)=c \Delta(z)^{\ord_{i\infty}(f)}\prod_{\z\in\SL_2(\Z)\backslash\H} \left(\Delta(z)\Big(j(z)-j(\z)\Big)\right)^{\frac{\ord_{\z}(f)}{\omega_{\z}}}
\]
with $c\in \C$. In particular, if $\ord_{i\infty}(f)=0$ and $f$ is normalized so that $c=1$,  then we see that
\begin{equation}\label{logfG}
\log\left(y^{\frac{k}{2}}|f(z)|\right)= \sum_{\z \in \SL_2(\Z)\backslash \H} \frac{\ord_\z(f)}{\omega_{\z}} \bgg_{\z}(z).
\end{equation}
By linearity of the inner product, it suffices to prove \eqref{eqn:jnbggz}, because Theorem \ref{thm:jninnergen} follows from \eqref{logfG} and \eqref{eqn:jnbggz} together with the valence formula.
\rm

We take $\mathbb{F}_1:=\JJ_n$ and $\mathbb{F}_2:=\GG_{\z}$ in Lemma \ref{lem:innerggen} and evaluate the contribution from the three terms in \eqref{eqn:innerggen}. We have $\mathbbm{f}_1=\jj_n$, $\mathcal{F}_1=-\calJ_n$, $f_1=j_n$,  $\mathbbm{f}_2=\bgg_{\z}$, $\mathcal{F}_2=\mathcal{G}_{\z}$, and $C=6$.

We first compute the contribution to the three terms in \eqref{eqn:innerggen} along the boundary near $i\infty$:
\begin{equation}\label{eqn:boundaryinfty}
\left\{x+iT: -\frac{1}{2}\leq x\leq \frac{1}{2}\right\}.
\end{equation}
Since the integral is oriented counter-clockwise from $\frac{1}{2}$ to $-\frac{1}{2}$, these yield the negative of the constant terms of the corresponding Fourier expansions. By Lemma \ref{lem:Fourierexps}, the first term equals
\begin{multline}\label{eqn:innerinfty1}
\sum_{m\geq 1} \overline{c_{\bgg_{\z}}^{++}(m) c_{\calJ_n}^{-}(-m)}W_2(-2\pi mT) -6\delta_{n=0}\frac{\log(T)}{T}\\
+ \sum_{m\leq -1}\overline{c_{\bgg_{\z}}^{+-}(m)}\left(\overline{c_{\calJ_n}^+(-m)}+4\pi n \delta_{m=-n}
\delta_{n\neq 0}
W_2(2\pi nT)\right) W_0(2\pi mT)+ 6\overline{c_{\calJ_n}^+(0)}\log(T).
\end{multline}
By Lemma \ref{lem:Wgrowth}, every term in \eqref{eqn:innerinfty1} other than the term with $\log(T)$ vanishes as $T\to\infty$.

By Lemma \ref{lem:Fourierexps}, the contribution from \eqref{eqn:boundaryinfty} from the second term in \eqref{eqn:innerggen} equals
\begin{equation}\label{eqn:innerinfty2}
-4\pi \sum_{m\geq 1}m \overline{c_{\bgg_{\z}}^{+-}(-m)}\left(c_{\jj_n}^{+-}(-m)W_0(-2\pi mT)+2\delta_{m=-n}\bm{\W}_{0}(-2\pi nT)\right)+6\delta_{n=0}\frac{1+\log(T)}{T}+6 c_{\jj_n}^{+-}(0).
\end{equation}
By Lemma \ref{lem:Wgrowth}, every term in \eqref{eqn:innerinfty2} other than $6c_{\jj_n}^{+-}(0)$ vanishes as $T\to\infty$.

The third integral yields $6$ times the complex conjugate of the constant term of $\JJ_n$, which has the shape in \eqref{eqn:Fourierconstant}. Since $\kappa=2>1$, every term in \eqref{eqn:Fourierconstant} other than
\begin{equation}\label{eqn:innerinfty3}
6\overline{c_{\JJ_n}^{+++}(0)} + 6\overline{c_{\JJ_{n}}^{+--}(0)}\log(T)
\end{equation}
vanishes as $T\to\infty$. However, since $\xi_{2}\!\left(\log(y)\right) = y$ and $\xi_0(y)=1$, we conclude from \eqref{eqn:Jxi} that
\begin{equation}\label{duality}
c_{\JJ_{n}}^{+--}(0)=\overline{c_{\jj_n}^{+-}(0)} = -c_{\calJ_{n}}^+(0).
\end{equation}
Combining \eqref{eqn:innerinfty1}, \eqref{eqn:innerinfty2}, \eqref{eqn:innerinfty3}, and \eqref{duality}, we conclude that the limit $T\to \infty$ of the contribution to \eqref{eqn:innerggen} from the integral along \eqref{eqn:boundaryinfty} is overall
\begin{equation}\label{eqn:innerinfty}
6
\overline{c_{\JJ_n}^{+++}(0)}+
6
c_{\jj_n}^{+-}(0).
\end{equation}

We next compute the contribution from the integral along $\partial\!\left( \mathcal{B}_{\varepsilon}(\z)\cap\mathcal{F}\right)$. Note that for a function $F$ satisfying weight two modularity we have
\begin{equation*}
\int_{\partial\left(\mathcal{B}_{\varepsilon}(\z)\cap \mathcal{F}\right)} F(z) dz = \frac{1}{\omega_{\z}}\int_{\partial\mathcal{B}_{\varepsilon}(\z)}F(z)dz.
\end{equation*}
Moreover a straightforward calculation yields that, for $\ell\in \Z$,
\begin{equation}\label{int2}
\frac{1}{2\pi i}\int_{\partial \mathcal{B}_{\varepsilon}(\z)} (z-\overline{\z})^{-2}X_{\z}^{\ell}(z)dz=
        \begin{cases}
        (\z-\overline{\z})^{-1}&\text{if }\ell=-1,\\
        0 &\text{otherwise},
        \end{cases}
\end{equation}
where the integral is taken counter-clockwise. Thus we need to determine the coefficient of $X_{\z}^{-1}(z)$ of the corresponding elliptic expansions. Noting that the orientation of \eqref{int2} is the opposite of the orientation in \eqref{eqn:innerggen}, Lemma \ref{lem:ellexps} implies that  the contribution along $\partial\!\left(\mathcal{B}_{\varepsilon}(\z)\cap\mathcal{F}\right)$ from the first term in \eqref{eqn:innerggen} is
\begin{align*}%\label{eqn:innerz1}
&\frac{4\pi}{\omega_{\z}}\!\Bigg(\sum_{m\geq 0} \overline{c_{\bgg_{\z},\z}^{++}(m)c_{\calJ_{n},\z}^{-}(-m-1)}\beta_{t_0}\!\Big(1-\varepsilon^2;-1,m+1\Big)\\
&\hspace{0.8cm}+ \sum_{m\leq -1}\!\! \overline{c_{\bgg_{\z},\z}^{+-}(m)c_{\calJ_n,\z}^{+}(-m-1)}\beta_{t_0}\!\left(1-\varepsilon^2;1,-m\right)\\
&\hspace{0.8cm}+\Big(\overline{c_{\bgg_{\z},\z}(0)}\beta_{t_0}\left(1-\varepsilon^2;1,0\right)+6\log\!\left(1-\varepsilon^2\right)\Big)\overline{c_{\calJ_n,\z}^{-}(-1)}\beta_{t_0}\!\left(1-\varepsilon^2;-1,1\right)\Bigg).
\end{align*}
By Lemma \ref{lem:betabnd}, each of these terms vanishes as $\varepsilon\to 0^+$.

By Lemma \ref{lem:ellexps}, the integral along $\partial(\mathcal{B}_{\varepsilon}(\z)\cap\mathcal{F})$ from the second term in \eqref{eqn:innerggen} evaluates as
\[%begin{multline}\label{eqn:innerz2}
\frac{4\pi}{\omega_{\z}}\Bigg(\left(-\frac{\omega_{\z}}{2}+\frac{6\varepsilon^2}{1-\varepsilon^2}\right)\left(c_{\jj_n,\z}^{++}(0)+c_{\jj_n,\z}^{--}(0)\bm{\beta}_{-1,0}(\varepsilon)\right)+\sum_{m\leq -1}\overline{c_{\bgg_{\z},\z}^{+-}(m)}c_{\jj_n,\z}^{+-}(m)\beta_{t_0}\!\left(1-\varepsilon^2;1,-m\right)\Bigg).
\]%end{multline}
Lemma \ref{lem:betabnd} implies that this converges to $-2\pi c_{\jj_n,\z}^{++}(0)$ as $\varepsilon\to 0^+$. Taking the limit $r=\varepsilon\to 0^+$ in the expansion of $\jj_n$ in Lemma \ref{lem:ellexps} (1) and using the bounds in Lemma \ref{lem:betabnd} and Corollary \ref{cor:boldbetabnd} then gives $c_{\jj_n,\z}^{++}(0)=\jj_n(\z)$.

Since $\JJ_n$ does not have a singularity at $\z$, there is no contribution from the third integral in \eqref{eqn:innerggen}. Hence we conclude that the integral along $\partial(\mathcal{B}_{\varepsilon}(\z)\cap\mathcal{F})$ altogether contributes $-2\pi \jj_n(\z)$. Thus the overall inner product is, using \eqref{eqn:innerinfty},
\begin{equation}\label{Gzj}
\langle j_n,\bgg_{\z}\rangle =  6\overline{c_{\JJ_n}^{+++}(0)}+6 c_{\jj_n}^{+-}(0)-2\pi\jj_n(\z) = -2\pi\jj_n(\z)+c_n.
\end{equation}
\end{proof}

\section{A generating function and the proof of Theorem \ref{thm:Fdivgen} and Corollary \ref{cor:Fdiv}}\label{sec:Fdiv}
The functions needed for Theorem \ref{thm:Fdivgen} are
\begin{equation}\label{eqn:HHdef}
\widehat{\HH}_{\z}:=-\frac{1}{2\pi}\GG_{\z}  -\frac{1}{4\pi}\mathbb E,\qquad \widehat{\II}_{\z}:=\GG_{\z}.
\end{equation}

\begin{proof}[Proof of Theorem \ref{thm:Fdivgen}]
We compute the inner product $\left<j_n,\bgg_{\z}\right>$ in another way and then compare with the evaluation in Theorem \ref{thm:jninnergen}. Namely, we apply Stokes' Theorem' with \eqref{eqn:HGxi} to instead write
\[
\left<j_n,\bgg_{\z}\right>_{T,\varepsilon} = \overline{\left<\xi_{2}\!\left(\GG_{\z}\right) ,j_n\right>_{T,\varepsilon}} = -\int_{\partial \mathcal{F}_{T,\z,\varepsilon}}\GG_{\z}(z)j_{n}(z)dz.
\]
By Lemma \ref{lem:Fourierexps}, the integral along the boundary \eqref{eqn:boundaryinfty} near $i\infty$ contributes
\[
c_{\GG_{\z}}^{+++}(n)+c_{\GG_{\z}}^{+--}(n)
\delta_{n\neq 0}
 \bm{\W}_{2}(2\pi nT)-\delta_{n=0}\frac{6}{T}(1+\log(T)) + \sum_{m\geq 1}c_{j_n}(m)c_{\GG_{\z}}^{++-}(-m)W_2(-2\pi mT).
\]
By Lemma \ref{lem:Wgrowth}, every term other than $c_{\GG_{\z}}^{+++}(n)$ vanishes as $T\to\infty$.

We then use the elliptic expansions of $j_n$ and $\GG_{\z}$ in Lemma \ref{lem:ellexps}, Lemma \ref{lem:betabnd}, and Corollary \ref{cor:boldbetabnd} to show that the integral along $\partial\!\left(\mathcal{B}_{\varepsilon}(\z)\cap\mathcal{F}\right)$ vanishes. Thus
\begin{equation}\label{eqn:jninnerotherway}
\left<j_n,\bgg_{\z}\right> = c_{\GG_{\z}}^{+++}(n).
\end{equation}
Taking the generating function of both sides of \eqref{eqn:jninnerotherway} and recalling \eqref{Hz}, we see that
\begin{equation*}
\sum_{n\geq 0} \left<j_n,\bgg_{\z}\right>q^n = \GG_{\z}^{+++}(z)=\widehat{\II}_{\z}^{+++}(z).
\end{equation*}

It remains to show that $\HH_{\z}$ is the holomorphic part of $\widehat{\HH}_{\z}$, which is a weight two sesquiharmonic Maass form because the biharmonic part of $2\GG_{\z}+\mathbb{E}$ vanishes by Lemma \ref{lem:Fourierexps}. For this, we plug \eqref{Gzj} into \eqref{eqn:jninnerotherway} and use \eqref{duality} to obtain
\begin{equation}\label{eqn:boldJncoeff}
\jj_n(\z)=-\frac{1}{2\pi}c_{\GG_{\z}}^{+++}(n)+\frac{3}{\pi}\overline{c_{\JJ_n}^{+++}(0)} -\frac{3}{\pi} \overline{c_{\calJ_n}^{+}(0)}.
\end{equation}
Again by \eqref{Hz}, to show that $\HH_{\z}$ is the holomorphic part of $\widehat{\HH}_{\z}$ it remains to prove that $12\overline{c_{\calJ_n}^{+}(0)}-12\overline{c_{\JJ_n}^{+++}(0)}$ is the $n$-th Fourier coefficient of the holomorphic part of $\mathbb{E}$. In order to see this, we next compute $\left<j_n,\mathcal{E}\right>_{T}$. Using Stokes' Theorem and noting \eqref{eqn:xiboldE} gives
\begin{equation}\label{jE}
\langle j_n, \mathcal{E}\rangle_{T} = \overline{\langle\xi_2(\mathbb{E}),j_n\rangle_{T}} = -\int_{\partial \mathcal F_T}\mathbb E(z)j_n(z) dz.
\end{equation}
Plugging in Lemma \ref{lem:Fourierexps}, we obtain that \eqref{jE} equals
\begin{multline}\label{Ejfirst}
c_{\mathbb E}^{+++}(n)+c_{\mathbb{E}}^{+--}(n)
\delta_{n\neq 0}
\bm{\W}_{2}(2\pi nT) + \delta_{n=0}\left( c_{\mathbb{E}}^{+--}(0)\log(T) + \frac{12}{T}\left(1+\log(T)\right)\right)\\
+\sum_{m\geq 1} c_{j_n}(m)c_{\mathbb E}^{++-}(-m) W_2(-2\pi mT).
\end{multline}
We see by Lemma \ref{lem:Wgrowth} that \eqref{Ejfirst} becomes
\begin{equation}\label{eqn:innerEisen1}
\left<j_n,\mathcal{E}\right>_T =c_{\mathbb E}^{+++}(n)+\delta_{n=0} c_{\mathbb{E}}^{+--}(0)\log(T)+O\!\left(\tfrac{\log(T)}{T}\right).
\end{equation}

We next use Lemma \ref{lem:innerggen} with $\mathbb{F}_1=\JJ_n$ and $\mathbb{F}_2=\mathbb{E}$  to compute the inner product another way. We have $\mathbbm{f}_1=\jj_n$, $\mathcal{F}_1=-\mathcal{J}_n$, $f_1=j_n$, $\mathbbm{f}_2=\mathcal{E}$ by \eqref{eqn:xiboldE}, $\mathcal{F}_{2}=\xi_0(\mathcal{E})$, and $C=-12$ in this case. We plug in Lemma \ref{lem:Fourierexps} to the first term from Lemma \ref{lem:innerggen} to obtain
\begin{multline*}
\sum_{m\geq 1}\overline{c_{\calJ_n}^+(m)} \  \overline{c_{\mathcal{E}}^{+-}(-m)} W_0(-2\pi mT) +4\pi n W_{2}(2\pi n T) \overline{ c_{\mathcal E}^{+-}(-n)} W_{0}(-2\pi nT)\\
+\left( \overline{c_{\calJ_n}^+(0)}-\delta_{n=0}\frac{1}{T}\right) \left(4\pi T-12\log(T)\right) + \sum_{m\leq -1}\overline{c_{\calJ_n}^-(m)} W_2(2\pi mT) \overline{c_{\mathcal E}^{++}(-m)}.
\end{multline*}
By Lemma \ref{lem:Wgrowth}, this becomes
\begin{equation}\label{vanish1}
-4\pi \delta_{n=0}+\overline{c_{\calJ_n}^+(0)}\!\left(4\pi T-12\log(T)\right) +O\!\left(\tfrac{\log(T)}{T}\right).
\end{equation}
By Lemma \ref{lem:Fourierexps}, the second term in Lemma \ref{lem:innerggen} equals
\begin{multline*}
4\pi \sum_{m\leq -1} c_{\jj_n}^{+-}(m) c_{E_2}(-m) W_0(2\pi mT) + 8\pi\delta_{n\neq 0} \bm{\W}_0(-2\pi nT) c_{E_2}(n)\\
 + \left(c_{\jj_n}^{+-}(0) T + \delta_{n=0}(1+\log(T))\right)\left(4\pi c_{E_2}(0)-\tfrac{12}{T}\right).
\end{multline*}
Using Lemma \ref{lem:Wgrowth}, \eqref{duality},  and the fact that $c_{E_2}(0)=1$, this becomes
\begin{equation}\label{eqn:innerEsecond}
-4\pi \overline{c_{\calJ_n}^+(0)}T+12\overline{c_{\calJ_n}^+(0)}+4\pi \delta_{n=0}\log(T)+4\pi\delta_{n=0} +O\!\left(\tfrac{\log(T)}{T}\right).
\end{equation}
By \eqref{eqn:innerinfty3} the contribution from the last term in Lemma \ref{lem:innerggen} is, using \eqref{duality}
\begin{equation}\label{eqn:innerEthird}
-12 \overline{c_{\JJ_n}^{+++}(0)} + 12 \overline{c_{\calJ_n}^+(0)} \log(T)+o(1).
\end{equation}
Combining the respective evaluations \eqref{vanish1}, \eqref{eqn:innerEsecond}, and \eqref{eqn:innerEthird}, we get
\begin{align}\label{combining}
\left<j_n,\mathcal{E}\right>_T&=4\pi \delta_{n=0} \log(T)+ 12 \overline{c_{\calJ_n}^+(0)} - 12\overline{c_{\JJ_n}^{+++}(0)} +o(1).
\end{align}
Comparing the constant terms in the asymptotic expansions of \eqref{combining} and \eqref{eqn:innerEisen1} gives that
\begin{equation*}
c_{\mathbb E}^{+++}(n) =12 \overline{c_{\calJ_n}^{+}(0)} - 12 \overline{c_{\JJ_n}^{+++}(0)}.
\end{equation*}
Plugging this into \eqref{eqn:boldJncoeff}, we conclude that $\jj_n(\z)$ is the $n$-th coefficient of $\widehat{\HH}_{\z}$, as claimed.
\end{proof}
We conclude with the proof of Corollary \ref{cor:Fdiv}.
\begin{proof}[Proof of Corollary \ref{cor:Fdiv}]
The claim follows from Theorem \ref{thm:Fdivgen}, \eqref{eqn:HHdef}, \eqref{eqn:HGxi}, \eqref{eqn:xiboldE}, \eqref{logfG}, and the valence formula.
\end{proof}

\end{document}